\documentclass[11pt]{article}
\usepackage[margin=1in]{geometry}
\usepackage[utf8]{inputenc}
\usepackage[dvipsnames,svgnames]{xcolor}
\usepackage{fullpage}
\usepackage{enumerate}
\usepackage{amsthm,amsmath, amssymb, amsfonts,multicol}
\usepackage{authblk}
\usepackage{verbatim}
\usepackage{tikz}
\usetikzlibrary{fit,matrix,positioning,shapes}
\usetikzlibrary{decorations.shapes}
\usepackage{graphicx,float}
\usepackage{mathtools}
\usepackage{hyperref}
\hypersetup{
    colorlinks=true,
    linkcolor=DarkRed, 
    urlcolor=DarkBlue,  
    citecolor=DarkGreen 
}
\providecommand{\keywords}[1]{\textbf{\textit{Key words---}} #1}

\theoremstyle{plain}
\newtheorem{thm}{Theorem}[section]
\newtheorem{lemma}[thm]{Lemma}
\newtheorem{prop}[thm]{Proposition}
\newtheorem{cor}[thm]{Corollary}
\newtheorem{conj}[thm]{Conjecture}
\newtheorem{remark}[thm]{Remark}
\newtheorem*{thm*}{Theorem}
\newtheorem*{lemma*}{Lemma}
\newtheorem*{prop*}{Proposition}
\newtheorem*{cor*}{Corollary}
\newtheorem*{conj*}{Conjecture}

\theoremstyle{definition}
\newtheorem{defn}[thm]{Definition}
\newtheorem*{defn*}{Definition}
\newtheorem{example}[thm]{Example}

\theoremstyle{remark}

\newcommand{\demph}[1]{\textcolor{DarkBlue}{\emph{#1}}}

\newcommand{\ind}{\mbox{$\perp \kern-5.5pt \perp$}}

\def\mult#1{\text{Mult}(#1)}
\def\multiset#1#2{\ensuremath{\left(\kern-.3em\left(\genfrac{}{}{0pt}{}{#1}{#2}\right)\kern-.3em\right)}}
\DeclareMathOperator{\St}{ST}
\DeclareMathOperator{\st}{st}
\DeclareMathOperator{\Tst}{Tst}
\DeclareMathOperator{\Tchi}{T\chi}
\def\rcomp#1#2{#1_{#2\rightarrow}}
\def\lcomp#1#2{#1_{\leftarrow#2}}
\DeclareMathOperator{\asc}{asc}
\DeclareMathOperator{\rev}{rev}
\newcommand{\rt}{r}

\usepackage[backend=bibtex,maxbibnames=15, isbn=false, doi=false, url=false, giveninits=true]{biblatex}
\title{The Total Chromatic Quasisymmetric Functions of a Graph}
\author{Laura Colmenarejo\footnote{\href{mailto: lcolmen@ncsu.edu}{lcomen@ncsu.edu}, Department of Mathematics, North Carolina State University, Raleigh, NC, 27605. \\ Partially supported by the Simons Foundation.} \ and Ian Klein\footnote{\href{mailto: iankleinmath@gmail.com}{iankleinmath@gmail.com}, Department of Mathematics, North Carolina State University, Raleigh, NC, 27605}}
\date{\empty}

\begin{document}

\maketitle

\begin{abstract}
   In this paper, we introduce and study two variants of the chromatic quasisymmetric function of a graph: the total chromatic quasisymmetric function via vertex labeling and via acyclic orientations. The original definition of the chromatic quasisymmetric function of a graph by Shareshian and Wachs depends on a labeling of the vertices of the graph, which directly affects the properties of the coefficients appearing in the decomposition of the chromatic quasisymmetric function of a graph into different bases. Motivated by this, we construct the first variant of the chromatic quasisymmetric function of a graph by normalizing it with respect to all the labelings of the vertices. The second variant is motivated by the \emph{tree isomorphism conjecture} and is constructed in terms of acyclic orientations. 
   
   We investigate the properties of the coefficients in the expansion in the monomial quasisymmetric basis for both variants and provide a comparative analysis. Furthermore, we derive explicit formulas for the coefficients in the monomial decomposition of the two variants for the star graph. For the labeling-based variant, these coefficients arise from a binomial identity for which we provide a combinatorial proof.
\end{abstract}


\section{Introduction}

In 1995, Stanley~\cite{Stanley1995ASymmetricFunctionGeneralization} defined a generalization of the chromatic polynomial in the framework of symmetric functions. Given a (finite) graph $G=(V,E)$, the chromatic symmetric function of $G$ is defined as $\chi_G (x):= \sum_{\kappa}x^\kappa$, where the sum runs over all proper colorings of the vertices of the graph and $x^\kappa = x_1^{k_1}x_2^{k_2}\cdots$, with $k_i$ being the number of vertices colored $i$. Since then, the study of the chromatic symmetric function of graphs has been a source of interesting research within algebraic combinatorics and other related areas. Among the most interesting problems, we highlight the \emph{Stanley-Stembridge conjecture}, also known as the \emph{$e$-positivity conjecture}, which has been recently proved by Hikita. 
\begin{thm}[Stanley-Stembridge conjecture~\cite{hikita2024proofstanleystembridgeconjecture,hikita2025qtchromaticsymmetricfunctions}]\label{thm: Stanley-Stembridge conjecture}
    Let $G$ be a unit interval graph and consider the expansion of the chromatic symmetric function of $G$ in the $e$-basis, $\chi_G (x)= \sum_{\lambda \vdash n}a_\lambda e_\lambda$. Then, $\chi_G$ is $e$-positive; that is, $a_\lambda \in\mathbb{Z}_{\geq 0}$.
\end{thm}

In 2012, Shareshian and Wachs~\cite{shareshian2012chromaticquasisymmetricfunctionshessenberg} introduced a refinement of the chromatic symmetric function of a graph in the framework of quasisymmetric functions. Given a graph $G=(V,E)$ and a labeling of its vertices so that $V=\{1,2,\ldots, |V|\}$, the chromatic quasisymmetric function of a graph $G$ is defined as
$$
\chi_G(x;q) = \sum_{\kappa}q^{\asc(\kappa)}x^\kappa,
$$ 
where $\asc(\kappa)$ counts the number of edges $\{i,j\}$ such that $i < j$ and $\kappa(i) < \kappa(j)$, and the sum runs over all the proper colorings of the vertices of $G$. 
This refinement was introduced as a way to connect the $e$-positivity conjecture, and more broadly the chromatic symmetric functions framework, to the cohomology of the Hessenberg varieties. 

Shareshian and Wachs also introduced a quasisymmetric version of the Stanley-Stembridge conjecture.
\begin{conj}~\cite{shareshian2012chromaticquasisymmetricfunctionshessenberg}\label{conj: SW-conjecture}
    Let $G=(V,E)$ be a unit interval graph and consider the natural order in $V$. Then, $\chi_G(x;q)$ is symmetric and $e$-positive.
\end{conj}
Notably, Hikita's proof of the Stanley-Stembridge conjecture does not imply the Shareshian-Wachs conjecture, leaving the latter still open. However, there has been some progress made towards this conjecture, for example~\cite{abreu2020chromaticsymmetricfunctionsmodular, Colmenarejo_2023, MGP2, lee2022explicitformulasepositivitychromatic}.

The starting point of our paper is the dependency of $\chi_G(x;q)$ on \emph{a natural labeling of the vertices of $G$}. For many families, including unit interval graphs, there is a natural, canonical labeling. However, that is not the case for all graphs. Moreover, the coefficients appearing in $\chi_G(x;q)$, which are polynomials in $\mathbb{Z}[q]$, can be significantly different depending on the labeling. For instance, consider the following example. 

\begin{example}\label{ex: intro}
Consider the star graph $\St_5$ with the following two labelings:
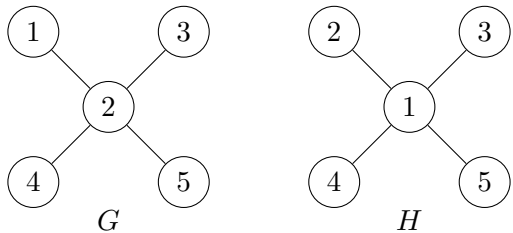
\begin{center}
\begin{tikzpicture}
    \node[shape=circle,draw=black] (2) at (0,0) {2};
    \node[shape=circle,draw=black] (1) at (-1,1) {1};
    \node[shape=circle,draw=black] (3) at (1,1) {3};
    \node[shape=circle,draw=black] (4) at (-1,-1) {4};
    \node[shape=circle,draw=black] (5) at (1,-1) {5};
    \node at (0,-1.5) {$G$};
    \path [-] (2) edge  (1);
    \path [-](2) edge (3);
    \path [-](2) edge (4);
    \path [-](2) edge (5); 
    \node[shape=circle,draw=black] (A) at (4,0) {1};
    \node[shape=circle,draw=black] (B) at (3,1) {2};
    \node[shape=circle,draw=black] (C) at (5,1) {3};
    \node[shape=circle,draw=black] (D) at (3,-1) {4};
    \node[shape=circle,draw=black] (E) at (5,-1) {5};
    \node at (4,-1.5){$H$};
    \path [-] (A) edge  (B);
    \path [-](A) edge (C);
    \path [-](A) edge (D);
    \path [-](A) edge (E); 
\end{tikzpicture}
\end{center}
Consider the expansion of $\chi_{G}(x;q)$ and $\chi_{H}(x;q)$ in the monomial quasisymmetric basis,  $\chi_{G}(q) = \sum_{\alpha}c_\alpha(q)M_\alpha$ and $\chi_{H}(q) = \sum_{\alpha}b_\alpha(q)M_\alpha$, respectively. These coefficients are displayed in Table~\ref{tab: example ST_5}, and we notice that the coefficients of $\chi_{H}(x;q)$ are more interesting from a combinatorial perspective than the coefficients of $\chi_{H}(x;q)$.
\begin{table}[ht]
    \centering
    \begin{tabular}{ |c|c|c| } 
 \hline
 $\alpha$ & $c_\alpha(q)$ & $b_\alpha(q)$\\
 \hline
$(1,1,1,1,1)$ & $6q^4 + 36q^3 + 36q^2 + 36q + 6$ & $24(q^4 + q^3 + q^2 + q + 1)$ \\
\hline
$(1,1,1,2)$ & $3q^4 + 18q^3 + 9q^2 + 6q$ & $12(q^4 + q^3 + q^2)$\\
\hline
$(1,1,2,1)$ & $3q^4 + 12q^3 + 9q^2+12q$ & $12(q^4 + q^3 + 1)$\\
 \hline
 $(1,2,1,1)$ & $12q^3 + 9q^2 + 12q + 3$ & $12(q^4 + q + 1)$\\
 \hline
  $(2,1,1,1)$ & $6q^3 + 9q^2 + 18q + 3$ & $12(q^2 + q + 1)$\\
 \hline
 $(1,1,3)$ & $q^4 + 4q^3 + 3q^2$ & $4(q^4 + q^3)$\\
 \hline 
  $(1,3,1)$ & $4q^3 + 4q$ & $4(q^4 + 1)$\\
 \hline
 $(3,1,1)$ & $3q^2 + 4q + 1$ & $4(q + 1)$\\
\hline
 $(1,2,2)$ & $6q^3$ & $6q^4$\\
 \hline 
  $(2,1,2)$ & $3q^3 + 3q$ & $6q^2$\\
\hline
$(2,2,1)$ & $6q$ & $6$\\
\hline
 $(1,4)$ & $q^3$ & $q^4$\\
 \hline
$(4,1)$ & $q$ & $1$\\
\hline
\end{tabular}
    \caption{Coefficients of the $M$-expansion of the chromatic quasisymmetric function for the star graph $\St_5$ with the two different labelings described in Example~\ref{ex: intro}.}
    \label{tab: example ST_5}
\end{table}

\end{example}

In this paper, we introduce and study two variants of the chromatic quasisymmetric function of a graph. The first variant is constructed from the chromatic quasisymmetric function of a graph by normalizing it with respect to all the labelings of the vertices of the graph. 
\begin{defn*} Let $G$ be a graph. We define the \demph{total labeling chromatic quasisymmetric function of $G$} as
$$\Tchi^{\mathcal{L}}_G(q) := \sum_{L}\chi_G(x;q),$$ 
where we sum over all possible labelings $L$ of the vertices of $G$.
\end{defn*} 
In Section~\ref{subsec: labelings}, we prove some properties of the coefficients, which are polynomials of $q$, appearing in the $M$-expansion of $\Tchi^{\mathcal{L}}_G(q)$ for any graph $G$. For instance, we show that the total labeling chromatic quasisymmetric function of a graph behaves well with respect to the disjoint union of graphs (Proposition~\ref{prop: Labeling Disjoint Union}), that the coefficients are palindromic and of maximal degree (Proposition~\ref{prop: palindrom Label}), and that they are invariant under reversing the indexing composition (Theorem~\ref{thm: reverse Labeling}). 
We attempted to find a variant of the deletion-contraction relation on $\Tchi_G^{\mathcal{L}}(q)$, but found some empirical evidence that makes us skeptical about the existence of one. 

In Section~\ref{subsec: total star labelings}, we study the total labeling CQF for the star graph by analyzing a normalized version of it (see Definition~\ref{def: Tstar}), for which the coefficients in its $M$-expansion are of the following form. 
\begin{thm*}
Let $\St_n$ be the star graph on $n$ vertices, and consider the $M$-expansion of $\Tst_n(q) = \sum_{\alpha}c_\alpha(q)M_\alpha$. Then, 
$$c_\alpha(q) = \sum_{i : \alpha_i = 1} \sum_{j\geq 0}\mult{\rcomp{\alpha}{i}}\mult{\lcomp{\alpha}{i}}\binom{n}{j}q^j[n-2j]_q,
$$ 
where $\mult{\beta}$ is essentially a multinomial coefficient, $\rcomp{\alpha}{i}$ and $\lcomp{\alpha}{i}$ are defined in terms of $\alpha$, and $[j]_q$ refers to the $q$-analog of the nonnegative integer $j$.
\end{thm*}
The combinatorial proof of this formula follows from a combinatorial argument in terms of the proper colorings of $\St_n$ and from the following binomial identity.
\begin{thm*}
Given $n$, $k$, and $s$, with $s\leq k$,
$$\sum_{j\geq 0} \binom{s + k -2j}{s-j}\binom{n-1-s - k +2j}{n-1-s-k +j} = \sum_{\ell=0}^s \binom{n}{\ell}.$$
\end{thm*}
The proof of this binomial identity turns out to be more challenging than expected, as a key part of its proof relies on showing the independence of the parameters $k$ and $s$, which do not appear in the coefficients on the right-hand side of the formula in the statement. Section~\ref{sec: binomial identity} is dedicated to presenting a combinatorial model and its proof. 

\bigskip

The second variant of the chromatic quasisymmetric function of a graph is motivated by a question that Stanley posted when introducing the chromatic symmetric function of a graph~\cite{Stanley1995ASymmetricFunctionGeneralization}. Stanley asked whether $\chi_G(x)$ distinguishes trees, and the positive answer to his question is known as the \emph{tree isomorphism conjecture}.
\begin{conj}[Tree isomorphism conjecture]\cite{Stanley1995ASymmetricFunctionGeneralization}\label{conj: tree isomorphism sym}
    Let $G$ and $H$ be non-isomorphic trees. Then, $\chi_G(x) \neq \chi_H(x)$.
\end{conj}
The conjecture has been computationally verified for all trees of up to 29 vertices~\cite{heil2018algorithmcomparingchromaticsymmetric}, but the main statement remains open. There has also been work done to verify this conjecture for certain classes of trees, 
for example~\cite{alisteprieto2012propercaterpillarsdistinguishedsymmetric, Huryn_2020,martin2007distinguishingtreeschromaticsymmetric,orellana2013graphsequalchromaticsymmetric}.

Motivated by this conjecture, the second variant of the chromatic quasisymmetric function of a graph is defined using the definition of the ascents in terms of orientations. Given a graph $G=(V,E)$, an orientation $\gamma$ on the edges $E$ is a function that assigns an orientation to each edge $\{u,v\}\in E$; that is $u \rightarrow v$ or $v \rightarrow u$. Then, we define the ascents of $G$ with respect to $\gamma$ as the number of edges $u\rightarrow v$ such that $\kappa(u) < \kappa(v)$. 
With this concept of ascents, which we denote by $\asc^\gamma(q)$, Ellzey~\cite{ellzey2017directedgraphgeneralizationchromatic, ellzey2017chromaticquasisymmetricfunctionsdirected} introduced a variant of the chromatic quasisymmetric function on graphs with respect to an orientation. The definition is exactly the same as for $\chi_G(x;q)$ with this variant of ascents, $\asc^\gamma(q)$, and we denote it by  $\chi_G^\gamma(x;q)$. Aval~\cite{Aval_2023} generalized the tree isomorphism conjecture in this framework. 
\begin{conj}~\cite{Aval_2023}\label{conj: tree isomorph qSym}
    Let $G$ and $H$ be non-isomorphic trees with an orientation. Then, $\chi_G^\gamma(x;q) \neq \chi_H^\gamma(x;q)$.
\end{conj}
Note that Aval's conjecture is stated in terms of directed graphs, which are essentially graphs with orientations.
Therefore, the second variant of the chromatic quasisymmetric function of a graph is constructed from the chromatic quasisymmetric function of a graph defined in terms of acyclic orientations, which are orientations that do not create cycles.
\begin{defn*}
    Let $G$ be a graph and $\Gamma$ be the set of acyclic orientations of $G$. Define the \demph{total orientation chromatic quasisymmetric function} of $G$ as 
    $$\Tchi_G^{\Gamma}(q) := \sum_{\gamma \in \Gamma}\chi_G^\gamma(x;q).$$
\end{defn*}

Similarly to our study of the total labeling chromatic symmetric function, in Section~\ref{subsec: acyclic orientations}, we prove some properties of the coefficients, which are polynomials of $q$, appearing in the $M$-expansion of $\Tchi^{\Gamma}_G(q)$ for any graph $G$. For instance, we show that it behaves well with respect to the disjoint union of graphs (Proposition~\ref{prop: Acyclic Disjoint Union}), that the coefficients are palindromic and of maximal degree (Theorem~\ref{thm: palindrom Acyclic}), and that they are invariant under reversing the indexing composition (Theorem~\ref{thm: reverese Acyclic}). Moreover, we also show that $\Tchi^{\Gamma}_G(q)$ satisfies a \emph{deletion-near-contraction} relation (Theorem~\ref{thm: Acyclic Del/nearcontraction}), like the one presented by Orellana et al~\cite{OrellanaStarBasis}. We also give a formula for $\Tchi_G^{\Gamma}(q)$ in terms of $\chi_G$ when $G$ is a tree, which also shows that $\Tchi_G^{\Gamma}(q)$ is a symmetric function when $G$ is a tree. By definition of the total orientation chromatic symmetric function, if $\Tchi_G^{\Gamma}(q)$ distinguishes trees, then so does $\chi_G$. Thus, there is a potential pathway to show that the quasisymmetric version of the tree isomorphism conjecture implies the symmetric version.

In this case, since the star graph is a tree, we give a formula for the coefficients appearing in the $m$-decomposition of the total orientation chromatic symmetric function of the star graph, which is a straightforward computation (Theorem~\ref{thm: TSto on trees}).

The outline of the paper is as follows: In Section~\ref{sec: background}, we introduce the background necessary for (quasi)symmetric functions, graphs, and chromatic (quasi)symmetric functions. Section~\ref{sec: TCQF by labelings and AO} is dedicated to introducing the total chromatic quasisymmetric functions for both variants. We also study some of the general properties and compare the two variants. In Section~\ref{sec: The star graph}, we look at the particular case of the star graph and compute explicit formulas for these total chromatic quasisymmetric functions. The formula for the total labeling chromatic quasisymmetric function reduces to a binomial identity. Finally, in Section~\ref{sec: binomial identity}, we present a model for the binomial identity and give a proof of it.

\section{Background}\label{sec: background}
In this section, we introduce the background on symmetric and quasisymmetric functions, on graphs, and on the theory of chromatic symmetric and quasisymmetric functions necessary for our study. For more details on the theory of symmetric and quasisymmetric functions, see~\cite{EC2}, and for more details on the theory of chromatic symmetric and quasisymmetric functions, see~\cite{shareshian2012chromaticquasisymmetricfunctionshessenberg}.

\subsection{Compositions and partitions}

A sequence $\alpha=(\alpha_1,\ldots, \alpha_\ell)$ is a composition of $n$, denoted $\alpha \models n$, if $\alpha_i \in \mathbb{Z}_{>0}$ and $|\alpha|:=\sum_i \alpha_i = n$. We refer to $\ell$ as the length of $\alpha$, $\ell(\alpha)$. If we allow zero entries, $\alpha_i\in \mathbb{Z}_{\geq 0}$, we say that $\alpha$ is a weak composition, and we denote by $\ell(\alpha)$ the number of nonzero entries. Moreover, if the entries in $\alpha$ are nonzero and in weakly decreasing order, that is $\alpha_1\geq \alpha_2\geq \ldots\geq \alpha_\ell>0$, we say that $\alpha$ is a partition of $n$, denoted $\alpha \vdash n$. Note that we usually denote (weak) compositions by $\alpha, \beta, \gamma$, and partitions by $\lambda, \mu, \nu$. Given a (weak) composition $\alpha = (\alpha_1, \ldots, \alpha_\ell)$, we denote by $\text{sort}(\alpha)$ the partition obtained by reordering the entries of $\alpha$ in weakly decreasing order and deleting the zero entries if necessary. Suppose that $\alpha = (\alpha_1, \ldots, \alpha_\ell)$ is a weak composition such that $\alpha_\ell\neq 0$; that is, $\alpha$ could have zero entries, but we do not consider the zero entries at the end. Then, we define the \demph{reverse composition} as $\alpha^{\rev} = (\alpha_\ell, \alpha_{\ell-1}, \ldots, \alpha_1)$.

\subsection{Symmetric and quasisymmetric functions}

Let $x=\{x_1,x_2,\ldots\}$ be an alphabet with infinitely many variables. A formal power series in $x$ over $\mathbb{Q}[q]$ is a function of the form $f(x;q) = \sum_{\alpha} c_\alpha(q) x^\alpha$, where the sum runs over all weak compositions $\alpha$, $c_\alpha\in \mathbb{Q}[x]$, and $x^\alpha = x_1^{\alpha_1}x_2^{\alpha_2}\cdots$. We say that $f(x;q)$ is homogeneous of degree $n$ if $c_\alpha(q) \neq 0$ for some $\alpha \models n$. Moreover, we denote by $[x^\alpha]f(x;q)$ the coefficient of $x^\alpha$ in $f(x;q)$, and for the coefficients $c_\alpha(q) \in \mathbb{Q}[q]$, we denote by $[q^k]c_\alpha(q)$ the coefficient of $q^k$ in $c_\alpha(q)$, for $k\in\mathbb{Z}_{\geq 0}$. Note also that we may omit the alphabet $x$ if it is clear from context, and write simply $f(q)$ or $f$ if the coefficients are in $\mathbb{Q}[q]$ or $\mathbb{Q}$, respectively.

We say that $f(x;q)$ is a symmetric function if it is invariant under the action of permutations; that is, $f(x_1,x_2, \ldots;q) = f(x_{\sigma(1)}, x_{\sigma(2)}, \ldots;q)$ for any permutation $\sigma$. We denote by $\Lambda$ the ring of symmetric functions on $x$ with coefficients in $\mathbb{Q}[q]$, and by $\Lambda^n$ the ring of homogeneous symmetric functions on $x$ of degree $n$ with coefficients in $\mathbb{Q}[q]$. Then, $\Lambda := \Lambda^0 \oplus \Lambda^1 \oplus \Lambda^2 \oplus \cdots$, where $\Lambda^0 = \mathbb{Q}[q]$.
As a vector space, $\Lambda^n$  has many interesting bases, all of them indexed by partitions of $n$. We introduce here the three bases that are more relevant for our study. The monomial basis is defined as $m_\lambda = \sum_{\alpha} x^\alpha$, where the sum runs over all weak compositions $\alpha$ such that $\text{sort}(\alpha)=\lambda$. 
The elementary basis is defined as
    $$ 
    e_k = \sum_{i_1 < i_2 < \cdots < i_k}x_{i_1}x_{i_2}\cdots x_{i_k}, \qquad \text{ and } \quad 
    e_\lambda = e_{\lambda_1}e_{\lambda_2}\cdots e_{\lambda_{\ell(\lambda)}}.
    $$
The power sum basis is defined as
    $$
    p_{k} = \sum_{i}x_i^k, \qquad \text{ and } \quad p_\lambda = p_{\lambda_1}p_{\lambda_2}\cdots p_{\lambda_{\ell(\lambda)}}.
    $$

Next, we introduce the quasisymmetric functions. We say that a formal power series $f(x;q)$ over $\mathbb{Q}[q]$ is quasisymmetric if for any weak composition $\alpha=(\alpha_1, \ldots, \alpha_\ell)$, the coefficient of $x^\alpha$ equals the coefficient of $x_{i_1}^{\alpha_1}x_{i_2}^{\alpha_2}\cdots x_{i_\ell}^{\alpha_\ell}$ for any set of subindices $0<i_1<i_2<\cdots < i_\ell$. We denote the ring of quasisymmetric functions on $x$ with coefficients in $\mathbb{Q}[q]$ by $QSym$, and the ring of homogeneous quasisymmetric functions of degree $n$ by $QSym^n$. As for $\Lambda^n$, there are many interesting bases for quasisymmetric functions, all of them indexed by compositions. The monomial basis is the only one relevant for our study, and it is defined as 
$M_\alpha := \displaystyle{\sum_{i_1 < i_2< \cdots i_\ell}x_{i_1}^{\alpha_1}x_{i_2}^{\alpha_2}\cdots x_{i_\ell}^{\alpha_\ell}}$.

\subsection{Graphs}\label{subsec: graphs}

Consider a graph $G=(V(G),E(G))$, where $V(G)$ is the set of vertices and $E(G)$ is the set of edges, which are unordered pairs $e=\{u,v\}$ with $u,v\in V(G)$. We denote the graph as $G=(V,E)$ if $G$ is known by context, and the edges as $e=\{u,v\}$. Note that we only work with finite graphs, and so the sets of vertices and edges are finite.
If $e=\{u,v\}\in E(G)$, we say that $u$ and $v$ are adjacent. The degree of a vertex $v\in V(G)$, $\deg(v)$, is the number of edges $e=\{u,v\}\in E(G)$ for some $u\in V(G)$, and we say that $e=\{u,v\}\in E(G)$ is an internal edge if $\deg(u),\deg(v)\geq 2$.

Our \emph{favorite family of graphs} in this paper is the \demph{star graph} $\St_n$ defined as the graph with $n$ vertices and such that there is a vertex $\rt$ called the \demph{root} and all the edges are of the form $\rt v$ with $v\in V$, $v\neq \rt$. Note that despite the name for this special vertex, we do not consider $\St_n$ as a rooted tree. Also, we reserve the notation $\rt$ for the root of the star graph. 

We say that a graph $G=(V,E)$ is simple if it contains no edges of the form $\{v,v\}$, with $v\in V$, and no repeated edges. All the graphs considered in our study are simple.
Given $G=(V,E)$, we define a path as a sequence of distinct vertices $v_1, v_2, \ldots, v_k$ in $V$ such that $\{v_i,v_{i+1}\} \in E$ for $i=1,\ldots k-1$. A cycle is a path such that $v_1=v_k$. We say that $G$ is connected if there exists a path between any two vertices of the graph. We say that $G$ is a tree if it is a connected graph with no cycles, and a forest if we only require the graph not to have cycles.

We define an orientation of $G$ as a map $\gamma$ that assigns a direction to each edge $e\in E$. That is, for every $e=\{u,v\}$, we either assign $(u,v)$ ($u\rightarrow v$) or $(v,u)$ ($v\rightarrow u$). We say that an orientation is acyclic if the resulting graph has no cycles of the form $v_1, v_2, \ldots, v_k$ with $v_i\in V$ and $(v_i,v_{i+1})\in E$ for all $i=1,\ldots k-1$, and $v_1=v_k$. We denote by $\Gamma(G)$, or simply $\Gamma$, the set of all acyclic orientations of $G$.
Note that a graph with an orientation is also known as a directed graph. 

We consider two decorations on the vertices of the graph $G=(V,E)$. We define a labeling of the vertices of the $G$ as a bijective map $L: V \longrightarrow \{1,2,\ldots, |V|\}$. That is, a labeling of $G$ is an ordering of the vertices of $G$. 
We refer to $L(v)$ as the label of $v\in V$, and we denote by $\mathcal{L}(G)$, or simply $\mathcal{L}$, the set of all labelings of $G$.
We also define a proper coloring of $G$ as a map $\kappa: V \longrightarrow \mathbb{Z}_{>0}$ such that if $u,v$ are adjacent vertices, then $\kappa(u)\neq \kappa(v)$. We refer to $\kappa(v)$ as the color of the vertex $v$, and we denote by $\mathcal{K}(G)$, or simply $\mathcal{K}$, the set of all the proper colorings of $G$. We associate the proper coloring $\kappa$ with the weak composition given by the number of vertices colored $i$, that is $(\#\kappa^{-1}(1), \#\kappa^{-1}(2),\ldots)$. We abuse the notation and denote this weak composition by $\kappa$. We say a composition $\alpha$ is the \demph{associated composition of $\kappa$} if $\alpha$ is obtained from the weak composition $\kappa$ by deleting the zero entries in $\kappa$. 

For the star graph, we may need to keep track of the label of the root, and so we denote by \demph{$\St_n^\rt$} the star graph together with a labeling $L$ such that $L(\rt)=\rt$, where we abuse notation and denote by $\rt$ both the root and its labeling. Note that there are $(n-1)!$ such labelings. 

Next, we define a statistic on the graphs related to the proper colorings. Let $G$ be a simple graph, with a labeling $L$ and a proper coloring $\kappa$. We say that the edge $e=\{u,v\}\in E$ is an ascent if $L(u)<L(v)$ and $\kappa(u)<\kappa(v)$. We denote by \demph{$\asc^L(\kappa)$} the number of ascents in $G$ with respect to the labeling $L$. 
If $G$ is a simple graph with an orientation, we say that $e=(u,v)\in E$ is an ascent if $\kappa(u) < \kappa(v)$. We denote by \demph{$\asc^\gamma(\kappa)$} the number of ascents in $G$ with respect to the orientation $\gamma$.

Finally, we define some operations on graphs. The first operation is the disjoint union of two graphs $G$ and $H$, written $G\sqcup H$, which is the graph with set of vertices $V(G) \sqcup V(H)$ and set of edges $E(G) \sqcup E(H)$. 
The other operations are defined over a graph $G=(V,E)$, with a vertex $e=\{u,v\}\in E$. We define \demph{$G\backslash e$} as the graph obtained from $G$ by deleting the edge $e$. We define \demph{$G/e$} as the graph obtained from $G$ by \demph{contracting the edge} $e$; that is, deleting $e$, $u$, and $v$, adding a new vertex $v_e$, and replacing the edges of the form $\{u,w\}\in E$ and $\{v,w\}\in E$ with edges of the form $\{v_e,w\}$. In this case, we say that $G/e$ contracts $e$ into the contracted vertex $v_e$. Finally, we define \demph{$G \odot e$} as the graph obtained by first contracting $e=\{u,v\}$ into the contracted vertex $v_e$ and then attaching a vertex $v'$ to $v_e$, denoting this edge by $\ell_e=\{v',v_e\}$. Note that $v'$ is only attached to $v_e$ and it is called the \demph{near-contracted edge} of $G\odot e$ and the graph $G\odot e$ is called the \demph{near-contraction of $e$}. Note that if $e$ is not an internal edge (i.e., $\deg(u)=\deg(v)\geq 2$), $G\odot e=G$. This operation was first introduced by Aliste-Prieto, de Mier, Orellana, and Zamora in 2022~\cite{alisteprieto2022markedgraphschromaticsymmetric}. 

\begin{example}\label{example: graph operations}
In this example, we present the graph $G$ on the right, with the edge $e=\{u,v\}$, and the graphs $G\setminus e$, $G \odot e$, and $(G\odot e)\setminus \ell_e$ to its left. 
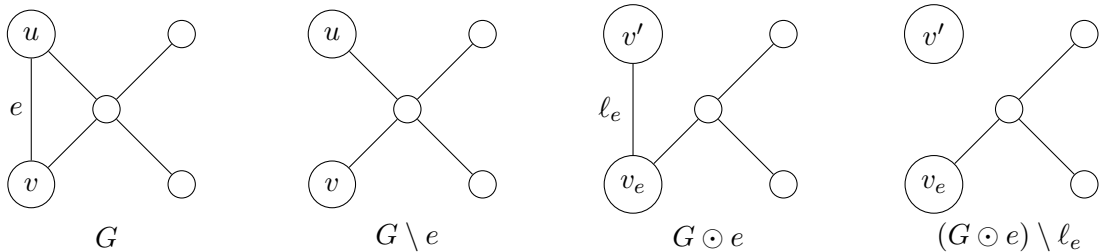
\begin{figure}[ht]
\centering
\begin{tikzpicture}
    \node[shape=circle,draw=black] (2) at (0,0) { };
    \node[shape=circle,draw=black] (1) at (-1,1) {$u$};
    \node[shape=circle,draw=black] (3) at (1,1) { };
    \node[shape=circle,draw=black] (4) at (-1,-1) {$v$};
    \node[shape=circle,draw=black] (5) at (1,-1) { };
    \node at (0,-1.7) {$G$};
    \node at (-1.2,0) {$e$};
    \path [-] (2) edge  (1);
    \path [-](2) edge (3);
    \path [-](2) edge (4);
    \path [-](2) edge (5); 
    \path [-](4) edge (1);

    \node[shape=circle,draw=black] (A) at (4,0) { };
    \node[shape=circle,draw=black] (B) at (3,1) { $u$};
    \node[shape=circle,draw=black] (C) at (5,1) { };
    \node[shape=circle,draw=black] (D) at (3,-1) { $v$};
    \node[shape=circle,draw=black] (E) at (5,-1) { };
    \node at (4,-1.7){$G \setminus e$};
    \path [-] (A) edge  (B);
    \path [-](A) edge (C);
    \path [-](A) edge (D);
    \path [-](A) edge (E);

\node[shape=circle,draw=black] (A) at (8,0) { };
    \node[shape=circle,draw=black] (B) at (7,1) { $v'$};
    \node[shape=circle,draw=black] (C) at (9,1) { };
    \node[shape=circle,draw=black] (D) at (7,-1) { $v_e$};
    \node[shape=circle,draw=black] (E) at (9,-1) { };
    \node at (8,-1.7){$G \odot e$};
    \node at (6.7, 0){$\ell_e$};
    \path [-] (D) edge  (B);
    \path [-](A) edge (C);
    \path [-](A) edge (D);
    \path [-](A) edge (E);

\node[shape=circle,draw=black] (A) at (12,0) { };
    \node[shape=circle,draw=black] (B) at (11,1) { $v'$};
    \node[shape=circle,draw=black] (C) at (13,1) { };
    \node[shape=circle,draw=black] (D) at (11,-1) { $v_e$};
    \node[shape=circle,draw=black] (E) at (13,-1) { };
    \node at (12,-1.7){$(G \odot e) \setminus \ell_e$};
    \path [-](A) edge (C);
    \path [-](A) edge (D);
    \path [-](A) edge (E);    
\end{tikzpicture}
    \label{fig: graphOperations}
\end{figure}
\end{example}

\subsection{Chromatic symmetric and quasisymmetric functions}

In~\cite{Stanley1995ASymmetricFunctionGeneralization}, Stanley introduced the chromatic symmetric function of a graph as a generalization of the chromatic polynomial.
Given a graph $G=(V,E)$, we define the chromatic symmetric function (CSF) of $G$ as  
$$
\chi_G (x) := \sum_{\kappa\in \mathcal{K}(G)}x^\kappa, \qquad\quad \text{ where } x^{\kappa} = x_1^{\#\kappa^{-1}(1)} x_2^{\#\kappa^{-1}(2)}\cdots.
$$ 

We note that $\chi_G(x)$ is a symmetric function over $\mathbb{Q}$ since permuting the variables corresponds to permuting the colors, and this does not affect the conditions on proper colorings. 

As symmetric functions, the CSF of graphs are very interesting. For instance, the CSF of a disjoint union of graphs equals the product of the CSF of each of the graphs~\cite[Proposition 2.3]{Stanley1995ASymmetricFunctionGeneralization}. Moreover, they essentially form a basis for the symmetric functions. To see this, we consider a family of graphs $\{G_k\}_{k\geq0}$, where $G_k$ is a connected graph in $k$ vertices, and for a partition $\lambda=(\lambda_1, \ldots,\lambda_\ell)$, we define $G_\lambda$ as the forest whose connected components are graphs in this family, $G_\lambda := G_{\lambda_1}\sqcup\ldots\sqcup G_{\lambda_\ell}$. Since the disjoint union of graphs corresponds to the product of the corresponding CSF of the graphs, we have that then, $\chi_{G_\lambda}(x)= \prod_k \chi_{G_{\lambda_k}}(x)$.

\begin{thm}~\cite[Theorem 5]{cho2016chromaticbasessymmetricfunctions}
Let $\{G_k\}_{k\geq0}$ be a family of connected graphs such that $G_k$ has $k$ vertices. Then, $\{\chi_{G_\lambda}(x)\}_{\lambda\vdash n}$ is a $\mathbb{Q}$-basis for $\Lambda^n$. Furthermore, the functions $\chi_{G_k}(x)$ are algebraically independent over $\mathbb{Q}$ and $\Lambda = \mathbb{Q}[\chi_{G_1}, \chi_{G_2}, \dots]$.
\end{thm}
This result tells us that any such family of graphs defines a basis of the symmetric functions. In~\cite{OrellanaStarBasis} the authors consider the family of star graphs, $G_k = \St_k$, so that the set of functions $\{\st_\lambda(x)\}_{\lambda\vdash n}$, defined by $\st_k (x)= \chi_{\St_{\lambda_k}}(x)$ and $\st_\lambda (x):= \prod_k\st_k(x)$, forms a basis of $\Lambda$, called the \demph{star basis}. This is a very interesting basis, and we highlight here the change of basis between the star basis and the power sum basis.  
\begin{thm}~\cite[Proposition 2.3]{OrellanaStarBasis}\label{thm: st to p basis}
For $n\geq 0$,
$$
\st_{n+1} (x)= \sum_{r=0}^n(-1)^r\binom{n}{r}p_{(r+1,1^{n-r})}(x)
\qquad\text{ and } \qquad
p_{n+1}(x)= \sum_{r=0}^n(-1)^r\binom{n}{r}\st_{(r+1,1^{n-r})}(x).
$$
\end{thm}
\begin{remark}
Note that we use uppercase letters to denote the star graph, $\St_n$, and lowercase letters for the CSF of the star graph, $\st_n= \chi_{\St_n}(x)$.
\end{remark}

The CSF of a graph does not satisfy the usual deletion-contraction formula that chromatic polynomials do. However, it satisfies a variant of it involving the near-contraction operation. 
\begin{thm}~\cite[Proposition 3.1]{alisteprieto2022markedgraphschromaticsymmetric}
   Let $G$ be a simple graph and consider an edge $e\in G$. Then, \begin{equation}\label{eq: del/contract sym}
   \chi_G(x) = \chi_{G\setminus e}(x) - \chi_{(G\odot e)\setminus \ell_e}(x) + \chi_{G\odot e}(x).
   \end{equation}
\end{thm}
As mentioned in~\cite{alisteprieto2022markedgraphschromaticsymmetric}, if $e$ is an internal edge, then the graphs represented in the RHS of~\eqref{eq: del/contract sym} all have one fewer edge than $G$. Thus, applying this relation to these three graphs, we can compute the chromatic symmetric function of $G$ as the sum of chromatic symmetric functions of graphs with 2 fewer internal edges than $G$. Since forests of star graphs are the only graphs with no internal edges, eventually this process leads to writing the chromatic symmetric function of any graph $G$ as the sum of products of the chromatic symmetric function of star graphs.

\bigskip

In 2014, Shareshian and Wachs~\cite{shareshian2016chromaticquasisymmetricfunctions} introduced a refinement of the CSF of a graph within the framework of quasisymmetric functions. 
\begin{defn}
Let $G$ be a graph with a labeling $L$. The chromatic quasisymmetric function (CQF) of $G$ is defined as
$$
\chi^L_G(x;q) := \sum_{\kappa \in \mathcal{K}(G)} q^{\asc^L(\kappa)}x^\kappa.
$$
\end{defn}
Their motivation came from the connection of the CQF of graphs and to the cohomology of the Hessenberg varieties.
As stated in the introduction, it is important to note that the analog of Stanley-Stembridge conjecture, known as \emph{Shareshian-Wachs conjecture} and stated in Conjecture~\ref{conj: SW-conjecture}, remains open. 

Now, the ascent statistic is also defined when we consider a graph with an orientation instead of a graph with a labeling (see Section~\ref{subsec: graphs}). Therefore, we also consider the following variant of the CQF of a graph. 
\begin{defn}
    Let $G=(V,E)$ be a simple graph with an orientation $\gamma$. The chromatic quasisymmetric function of $G$ is defined to be 
    $$\chi_G^\gamma (x;q) := \sum_{\kappa}q^{\asc^\gamma(\kappa)}x^\kappa.$$
\end{defn}
\begin{remark}
To simplify the notation, we drop the alphabet from the symmetric and quasisymmetric functions. For instance, $\chi_G$ denotes the CSF of the graph $G$ and $\chi_G(q)$ denotes the CQF of the graph $G$. 
\end{remark}

\section{The total chromatic quasisymmetric function of a graph}\label{sec: TCQF by labelings and AO}

As mentioned in the introduction, one of the key aspects of the CQF of a graph is its reliance either on the labeling of the vertices or on the acyclic orientation, as in both cases, it directly affects the number of ascents, and consequently, the coefficient of $x^\alpha$ (which is a polynomial in $\mathbb{Q}[q]$).

In this section, we introduce two new variants of the CSF of a graph and study their properties in general. In Section~\ref{subsec: labelings}, we define and study the variant with respect to the labelings, while in Section~\ref{subsec: acyclic orientations}, we do it for the variant with respect to the acyclic orientations. In Section~\ref{subsec: comparison}, we compare both variants. 

\subsection{Via labelings}\label{subsec: labelings}

We start by defining the variant with respect to the labeling.
\begin{defn} Let $G=(V,E)$ be a graph, and $\mathcal{L}$ and $\mathcal{K}$ be the sets of vertex labelings and proper colorings of $G$, respectively. We define the \demph{total labeling chromatic quasisymmetric function} of $G$ to be
$$\Tchi_G^{\mathcal{L}}(q) := \sum_{L \in \mathcal{L}} \chi_G^L(q) =  \sum_{L \in \mathcal{L}} \sum_{\kappa \in \mathcal{K}} q^{\asc^L(\kappa)}x^\kappa.$$
\end{defn}

\begin{example}\label{ex: C4 total labeling}
Let $G$ be the cycle with 4 vertices, and write $\Tchi_G^{\mathcal{L}}(q) = \sum_\alpha c^{\mathcal{L}}_\alpha(q)M_\alpha$. Then, we have the following coefficients $c^{\mathcal{L}}_\alpha(q)$:
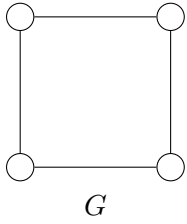
$$
\begin{tikzpicture}
    \node[shape=circle,draw=black] (1) at (-1,0) {};
    \node[shape=circle,draw=black] (3) at (1,0) {};
    \node[shape=circle,draw=black] (4) at (-1,-2) {};
    \node[shape=circle,draw=black] (5) at (1,-2) {};
    \node at (0,-2.5) {$G$};
    \path [-] (3) edge  (1);
    \path [-](5) edge (3);
    \path [-](5) edge (4);
    \path [-](1) edge (4); 
    \node[anchor=north west] at (3,0.5) {
$\begin{array}{ |c|c| } 
 \hline
 \alpha & c^{\mathcal{L}}_\alpha(q)\\
 \hline
(1,1,1,1) & 54q^4 + 128q^3 + 208q^2 + 128q +54 \\
\hline
(1,1,2) & 16q^4 + 16q^3 + 32q^2 + 16q +16 \\
\hline
(1,2,1) & 8q^4 + 16q^3 + 48q^2 + 16q + 8 \\
 \hline
 (2,1,1) & 16q^4 + 16q^3 + 32q^2 + 16q +16 \\
 \hline
 (2,2) & 8q^4 + 8q^3 + 16q^2 + 8q + 8 \\
 \hline 
\end{array}$};
\end{tikzpicture} 
$$

\end{example}

We start by noticing that the total labeling CQF of graphs behaves well for unions of graphs. 
\begin{prop}\label{prop: Labeling Disjoint Union}
Let $G=(V,E)$ and $G'=(V',E')$ be two graphs, and consider their disjoint union $H = G\sqcup G'$. Then, $$\Tchi_H^L(q)  = \binom{|V| + |V'|}{|V|}\Tchi_G^{\mathcal{L}}(q)\Tchi_{G'}^{\mathcal{L}}(q).$$
\end{prop}

\begin{proof}
The results follows from noticing that, given a labeling $L$ of $G$ and a labeling $L'$ of $G'$, there are $\binom{|V|+|V'|}{|V|}$ ways of labeling $H$ such that the relative order of vertex labels in the restriction of $H$ to $G$ and $G'$ match the labelings of $G$ and $G'$, respectively. 
\end{proof}

Our study of the total labeling CQS of a graph focuses on its expansion in the monomial quasisymmetric basis. That is, given a graph $G$, we study the coefficients $c^{\mathcal{L}}_\alpha(q)$ appearing in 
$$\Tchi_G^{\mathcal{L}}(q) =  \sum_{\alpha}c^{\mathcal{L}}_\alpha(q)M_\alpha,$$
where the sum runs over all compositions $\alpha$ of $n = |V|$. We refer to it as the \demph{$M$-expansion of the total labeling CQF of $G$}.

We start by noticing that the $q$-coefficients are symmetric. 
\begin{prop} \label{prop: palindrom Label}
    Let $G=(V,E)$ be a graph on $n=|V|$ vertices,  and consider the $M$-expansion of its total labeling CQF, $\Tchi_G^{\mathcal{L}}(q) =  \sum_{\alpha}c^{\mathcal{L}}_\alpha(q)M_\alpha$. Then,  $c^{\mathcal{L}}_\alpha(q)$ is a palindromic polynomial in $q$ with degree at most $|E|$, i.e., for all $k$, $$[q^k]c^{\mathcal{L}}_\alpha(q) = [q^{|E|-k}]c^{\mathcal{L}}_\alpha(q).$$
\end{prop}

\begin{proof}
 The bound for the degree of $c^{\mathcal{L}}_\alpha(q)$ follows from the definition of ascents and the fact that there cannot be more ascents than edges in the graph. 

 Now, fix $\kappa \in \mathcal{K}$ and $0\leq k \leq |E|$. Let $\mathcal{L}_k$ be the set of labelings $L\in \mathcal{L}$ such that $\asc^L(\kappa)=k$. We define the map $\phi_k : \mathcal{L}_k \longrightarrow \mathcal{L}_{|E|-k}$ as follows: Given a labeling $L\in \mathcal{L}_k$ and a vertex $v\in V$ with $L(v) =i$, for $1\leq i \leq n$, we define $\phi_k(L)$ as the labeling $L'$ such that $L'(v) = n-i$. Then, $\phi_k$ is well-defined and satisfies the following: an edge $e=\{u,v\}\in E$ contributes to $\asc^L(\kappa)$ if and only if the same edge $e=\{u,v\}\in E$ does not contribute to $\asc^{L'}(\kappa)$ with $L'=\phi_k(L)$. That is, given a proper coloring $\kappa$, the number of labelings of $G$ that yield $k$ ascents is the same as the number of labelings of $G$ that yield $|E|-k$ ascents. Thus, the result follows.
\end{proof}

We also notice that the $q$-coefficients are invariant under reversing the composition. 
\begin{thm}\label{thm: reverse Labeling}
    Let $G=(V,E)$ be a graph on $n=|V|$ vertices and consider the $M$-expansion of its total labeling CQF, $\Tchi_G^{\mathcal{L}}(q) = \sum_{\alpha}c^{\mathcal{L}}_\alpha(q) M_\alpha$. Then, for any composition $\alpha$ of $n$, $c^{\mathcal{L}}_\alpha(q) = c^{\mathcal{L}}_{\alpha^{\rev}}(q)$.
\end{thm}

\begin{proof}
Given $k$ and $\alpha$, let $A_{\alpha}^k$ be the set of pairs $(\kappa, L) \in \mathcal{K}\times \mathcal{L}$ such that $\text{sort}(\kappa)=\alpha$ and $\asc^L(\kappa)=k$. Note that $[q^k]c^{\mathcal{L}}_\alpha(q) = |A_{\alpha}^k|$.

We define a map $\phi: A_\alpha^k \to A_{\alpha^{\rev}}^{k}$ in the following way:
 Given $(\kappa,L)\in A_\alpha^k$, $\phi(\kappa,L) := (\kappa^r,L^r)$, where $\kappa^r(v) = \ell(\alpha)- \kappa(v) + 1$ and $L^r(v) = n - L(v) + 1$.  It is clear by definition that $\phi$ is a bijection. 
 
Consider an edge $e=\{u,v\} \in E$ that contributes to the ascents of $\kappa$ with labeling $L$. Assume without loss of generality that $L(u) < L(v)$ and $\kappa(u) < \kappa(v)$. Then, $L^r(u) > L^r(v)$ and $\kappa^r(u) > \kappa^r(v)$, and so $\{u,v\}$ contributes to the ascents of $\kappa^r$ with labeling $L^r$. Similarly, if $e=\{u,v\} \in E$ does not contribute to the ascents of $\kappa$ with labeling $L$, $e=\{u,v\}$ does not contribute to the ascents of $\kappa^r$ with labeling $L^r$. Thus, $\phi$ is a bijection that preserves the number of ascents, and so the result follows.
\end{proof}

\subsection{Via acyclic orientations}\label{subsec: acyclic orientations}

While the total labeling CQF of a graph seems natural to consider, its computation is fairly complex. Here, we consider another variant in terms of acyclic orientations since the set of acyclic orientations of a graph is smaller than the set of labelings. 

Recall that we can also consider the statistic of ascents for graphs with orientations, rather than labelings. 
Moreover, given a simple graph $G=(V,E)$, a labeling $L$ of its vertices induces an acyclic orientation $\gamma_L$ on $G$ by assigning the edge $\{i,j\}\in E$ the orientation $(i,j)$ if $L(i)<L(j)$, and $(j,i)$ otherwise. Now, several labelings may yield the same acyclic orientation since $\gamma_L$ only depends on the relative order of the adjacent vertices. Therefore, $\chi_G^L(q) = \chi_G^{\gamma_L}(q)$. Also, given an orientation $\gamma$, $\chi_G^\gamma(q) = \chi_G^L(q)$ for any labeling $L$ such that $\gamma_L=\gamma$.

Thus, we can say that while we can define the CQF of a graph $G$ under any orientation, the acyclic orientations are the ones that reflect the original idea of the CQF as introduced by Shareshian and Wachs. This leads us to the following definition.

\begin{defn}
    Let $G$ be a graph, and let $\Gamma$ be the set of acyclic orientations of $G$. Define the \demph{total orientation chromatic quasisymmetric function} of $G$ as 
    $$\Tchi^{\Gamma}_G(q) := \sum_{\gamma \in \Gamma} \chi_G^\gamma(q).$$
\end{defn}

\begin{example}\label{ex: C4 total orientation}
    Let $G$ be the cycle with 4 vertices, and write $\Tchi_G^{\Gamma}(q) = \sum_\alpha c^{\Gamma}_\alpha(q)M_\alpha$. Then, the coefficients $c^{\Gamma}_\alpha(q)$ are:
$$
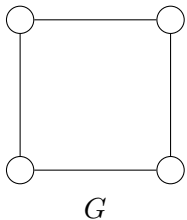
\begin{tikzpicture}
    \node[shape=circle,draw=black] (1) at (-1,0) {};
    \node[shape=circle,draw=black] (3) at (1,0) {};
    \node[shape=circle,draw=black] (4) at (-1,-2) {};
    \node[shape=circle,draw=black] (5) at (1,-2) {};
    \node at (0,-2.5) {$G$};
    \path [-] (3) edge  (1);
    \path [-](5) edge (3);
    \path [-](5) edge (4);
    \path [-](1) edge (4); 
\node[anchor=north west] at (3,0.5) {
$\begin{array}{ |c|c| } 
 \hline
\alpha & c^{\Gamma}_\alpha(q)\\
 \hline
(1,1,1,1) & 24q^4 + 88q^3 + 112q^2 + 88q^3 + 24q^4 \\
\hline
(1,1,2) & 4q^4 + 16q^3 + 16q^2 + 16q + 4\\
\hline
(1,2,1) & 4q^4 + 16q^3 + 16q^2 + 16q + 4\\
 \hline
(2,1,1) & 4q^4 + 16q^3 + 16q^2 + 16q + 4\\
 \hline
(2,2) & 2q^4 + 8q^3 + 8q^2 + 8q + 2 \\
 \hline 
 \end{array}$};
\end{tikzpicture}
$$

If we compare this to Example~\ref{ex: C4 total labeling}, we can see that $\Tchi_{C_4}^o(q)$ is symmetric, while $\Tchi_{C_4}^L(q)$ is not. Furthermore, each $c^{\Gamma}_\alpha(q)$ is a log-concave polynomial, while that is not true for $c^{\mathcal{L}}_\alpha(q)$ for $\alpha = (1,1,2),\ (1,2,1),\ (2,1,1),$ and $(2,2)$.

\end{example}

We start by noticing that the total orientation CQF of graphs also behaves well for disjoint unions of graphs. 
\begin{prop}\label{prop: Acyclic Disjoint Union}
    Let $G$ and $G'$ be two graphs and consider $H = G \sqcup G'$. Then, $$\Tchi^{\Gamma}_H (q)= \Tchi^{\Gamma}_G (q)\cdot \Tchi^{\Gamma}_{G'}(q).$$
\end{prop}

\begin{proof}
The result follows from noticing that there is a natural bijection between pairs of acyclic orientations of $G$ and $G'$ and acyclic orientations of $H$ that maps $(\gamma,\gamma') \longmapsto \gamma\sqcup\gamma'$. Therefore,
\begin{align*}
        \Tchi^{\Gamma}_G(q)\cdot \Tchi^{\Gamma}_{G'}(q) &= \left(\sum_{\gamma}\chi_G^\gamma(q)\right)\left(\sum_{\gamma'}\chi_{G'}^{\gamma'}(q)\right) = \sum_{(\gamma,\gamma')}\chi_G^\gamma(q)\cdot \chi_{G'}^{\gamma'}(q) =
        \sum_{\gamma,\gamma'}\chi_H^{\gamma\sqcup \gamma'}(q)= \Tchi^{\Gamma}_H(q).
    \end{align*}
\end{proof}

As for the labeling variant, our study of the total orientation CQS of a graph focuses on its expansion in the monomial quasisymmetric basis. That is, given a graph $G$, we study the coefficients $c^{\Gamma}_\alpha(q)$ appearing in 
$$\Tchi_G^{\Gamma}(q) =  \sum_{\alpha}c^{\Gamma}_\alpha(q)M_\alpha,$$
where the sum runs over all compositions $\alpha$ of $n = |V|$. We refer to it as the \demph{$M$-expansion of the total orientation CQF of $G$}.

\begin{prop}\label{prop: CSF coeffs in orientation total}
Let $G=(V,E)$ be a graph and consider the $M$-expansion of its total orientation CQF, $\Tchi_G^{\Gamma}(q) = \sum_{\alpha}c^{\Gamma}_\alpha(q) M_\alpha$. Consider also the $m$-expansion of the CSF of $G$,  $\chi_G = \sum_{\lambda}c_\lambda m_\lambda$. Then, $$[q^{|E|}]c^{\Gamma}_\alpha(q) = c_{sort(\alpha)}.$$
\end{prop}

\begin{proof}
Recall that $c_{sort(\alpha)}$ counts the number of proper colorings of $G$ with associated composition $\alpha$, and that $[q^{|E|}]c_\alpha(q)$ is the number of pairs $(\kappa,\gamma)$, where $\kappa$ is a proper coloring of $G$ with associated composition $\alpha$ and $\gamma$ is an acyclic orientation of $G$, such that the pair yields $|E|$ ascents.

We first notice that, given a coloring $\kappa$, the orientation $\gamma_\kappa$ defined by orienting the edge $\{u,v\}$ as $(u,v)$ whenever $\kappa(u)<\kappa(v)$ yields to exactly $|E|$ ascents. Moreover, we claim this orientation is acyclic. Suppose by contradiction that there exists a cycle $v_1v_2\ldots v_\ell v_1$ with $(v_i,v_{i+1})\in E$ for $1\leq i <\ell$ and $(v_\ell, v_1)\in E$. Then, looking at the proper coloring $\kappa$, we have that by definition of the orientation $\gamma_\kappa$, $\kappa(v_1)<\kappa(v_2)< \cdots < \kappa(v_\ell) < \kappa(v_1)$, which is a contradiction and our claim is proven. 
We also notice that any other orientation $\gamma\neq \gamma_\kappa$ has at least one edge that does not yield an ascent, and so it does not contribute to $[q^{|E|}]c^{\Gamma}_\alpha(q)$. 

Thus, for each proper coloring $\kappa$, there is a unique acyclic orientation $\gamma_\kappa$ that contributes to the coefficient $[q^{|E|}]c^{\Gamma}_\alpha(q)$, meaning that $[q^{|E|}]c^{\Gamma}_\alpha(q) = c_{sort(\alpha)}$ as stated in the result. 
\end{proof}

We now notice that the $q$-coefficients are symmetric. 
\begin{thm}\label{thm: palindrom Acyclic}
Let $G=(V,E)$ be a graph, and consider the $M$-expansion of its total orientation CQF, 
$\Tchi_G^{\Gamma}(q) = \sum_{\alpha}c_\alpha^o(q)M_\alpha$. Then, $c_\alpha^o(q)$ is a palindromic polynomial in $q$ of degree $m=|E|$.
\end{thm}

\begin{proof}
    Let $k \leq |E|$ and consider a coloring $\kappa$ such that its associated composition is $\alpha$ and an orientation $\gamma$ such that $\asc^\gamma(\kappa)=k$. There is a natural involution between acyclic orientations yielding $k$ ascents and orientations yielding $|E| - k$ orientations resulting from reversing the direction of each edge. This means that $[q^k]c_\alpha^o(q) = [q^{|E|-k}]c_\alpha^o(q)$. Furthermore, we know by Proposition~\ref{prop: CSF coeffs in orientation total} that $c^{\Gamma}_\alpha(q)$ has indeed degree $|E|$.
\end{proof}

Next, we notice that, as for the total labeling CQF, the $q$-coefficients are invariant under reversing the composition.
\begin{thm}\label{thm: reverese Acyclic}
Let $G=(V,E)$ be a graph, and consider the $M$-expansion of its total orientation CQF,
$\Tchi_G^{\Gamma}(q) = \sum_{\alpha}c_\alpha^o(q)M_\alpha$. Then, $c_\alpha^o(q) = c_{\alpha^{\rev}}^o(q)$.
\end{thm}

\begin{proof}
Consider a proper coloring $\kappa$ of $G$ with associated composition $\alpha$ and an orientation $\gamma$ such that they yield $k$ ascents and define $\kappa^r$ as the coloring of $G$ obtained by setting $\kappa^r(v)=\ell(\alpha) - \kappa(v) + 1$,  for all $v\in V$. Then $\kappa^r$ has associated composition $\alpha^{\rev}$ and if $\kappa(v) < \kappa(v')$, then $\kappa^r(v) > \kappa^r(v')$. This means that $\kappa^r$ yields $|E| - k$ ascents. This gives us a natural involution between proper colorings of $G$ with associated composition $\alpha$ yielding $k$ ascents and proper colorings of $G$ with associated composition $\alpha^{\rev}$ yielding $|E| - k$ ascents. Thus,  $[q^k]c_\alpha^o(q) = [q^{|E|}]c_{\alpha^{\rev}}^o(q)$, and by Theorem~\ref{thm: palindrom Acyclic}, 
$$[q^k]c_\alpha^o(q) = [q^{|E|}]c_{\alpha^{\rev}}^o(q) = [q^k]c_{\alpha^{\rev}}^o(q).$$ 
\end{proof}

The following result presents a \emph{deletion-near-contraction} property for the total orientation CQF, for which we have not found an analog in the total labeling CQF case. 
\begin{thm}\label{thm: Acyclic Del/nearcontraction}
    Let $G$ be a graph and let $e$ be an internal edge in $G$ that is not part of a cycle. Then,
    $$
    \Tchi^{\Gamma}_G = \Tchi^{\Gamma}_{G\odot e} + (1+q)\left(\Tchi^{\Gamma}_{G\setminus e} - \Tchi^{\Gamma}_{(G\odot e)\setminus \ell_e }\right).
    $$
\end{thm}

\begin{proof}
    First, we rewrite the identity in the statement as
    \begin{align}\label{eq: initial near-contraction}
    (1+q) \Tchi^{\Gamma}_{G\backslash e} - \Tchi^{\Gamma}_G = + (1+q) \Tchi^{\Gamma}_{(G\odot e)\setminus \ell_e } - \Tchi^{\Gamma}_{G\odot e}.
    \end{align}
    Next, we observe the following. Let $H$ be a graph and $h$ be an edge in $H$ that is not part of a cycle. Then, we have that an acyclic orientation of $H\backslash h$ induces two acyclic orientations in $H$, one for each possible orientation of $h$. Moreover, a proper coloring of $H\backslash h$ corresponds to a proper coloring of $H$ as long as the vertices corresponding to $h$ have different colors. Moreover, we note that for the two orientations, the coloring either keeps the same number of ascents or increases it by one. This observation applies to both $G$, with the edge $e$, and $G\odot e$, with the edge $\ell_e$. It implies that the remaining terms on each side of~\eqref{eq: initial near-contraction} are those corresponding to colorings with the vertices of $e$ and $\ell_e$, respectively, having the same color in $G\backslash e$ and $(G\odot e) \backslash \ell_e$, respectively, with the $(1+q)$ factor. That is, if $e=\{u,v\}$ and $\ell_e = \{v',v_e\}$, we have that 
    \begin{align*}
    (1+q) \Tchi^{\Gamma}_{G\backslash e} - \Tchi^{\Gamma}_G    &= (1+q) \sum_{\gamma \in \Gamma(G\backslash e)} \sum_{\substack{\kappa \in \mathcal{K}(G\backslash e) \\ \kappa(u)=\kappa(v)}} q^{\asc^\gamma(\kappa)}x^\kappa, \qquad \text{ and } \\
    (1+q) \Tchi^{\Gamma}_{(G\odot e)\backslash \ell_e} - \Tchi^{\Gamma}_{G\odot e}    &= (1+q) \sum_{\gamma \in \Gamma((G\odot e)\backslash \ell_e)} \sum_{\substack{\kappa \in \mathcal{K}((G\odot e)\backslash \ell_e) \\ \kappa(v')=\kappa(v_e)}} q^{\asc^\gamma(\kappa)}x^\kappa,
    \end{align*}
    and so Equation~\eqref{eq: initial near-contraction} follows if we can show that 
    \begin{align*}
        \sum_{\gamma \in \Gamma(G\backslash e)} \sum_{\substack{\kappa \in \mathcal{K}(G\backslash e) \\ \kappa(u)=\kappa(v)}} q^{\asc^\gamma(\kappa)}x^\kappa = \sum_{\gamma \in \Gamma((G\odot e)\backslash \ell_e)} \sum_{\substack{\kappa \in \mathcal{K}((G\odot e)\backslash \ell_e) \\ \kappa(u)=\kappa(v)}} q^{\asc^\gamma(\kappa)}x^\kappa.
    \end{align*}
    We show this identity by giving a bijection $\varphi$ from pairs $(\gamma,\kappa)$ with $\gamma \in \Gamma(G\backslash e)$ and $\kappa \in \mathcal{K}(G\backslash e)$ such that $\kappa(u)=\kappa(v)$ to pairs $\varphi(\gamma,\kappa) = (\kappa',\gamma')$ with $\gamma' \in  \Gamma((G\odot e)\backslash \ell_e)$ and $\kappa' \in \mathcal{K}((G\odot e)\backslash \ell_e)$ such that $\kappa'(v')= \kappa'(v_e)$, where 
    \begin{itemize}
        \item $\gamma'$ is defined from $\gamma$ by keeping the same orientation on those edges not involving $u$ and $v$ in $G\backslash e$, and orienting the edges $\{v_e,w\}$ in $\gamma'$ with the same orientation as the corresponding edge $\{u,w\}$ or $\{v,w\}$ in $\gamma$, and 
        \item $\kappa'$ is defined from $\kappa$ as the coloring obtained by keeping all the vertices $w\notin \{u,v\}$ with the same color, $\kappa'(w) = \kappa(w)$, and $\kappa'(v_e):=\kappa(u) = \kappa(v) =: \kappa'(v')$. 
    \end{itemize}
    The map $\varphi$ is a well-defined bijection such that $\asc^\gamma(\kappa)=\asc^\gamma(\kappa')$ and $x^\kappa = x^{\kappa'}$. Thus, the result follows. 
\end{proof}

We finish this section by analyzing a particular case in which $\Tchi^{\Gamma}_G$ is symmetric and conjecturing a more general result. 
\begin{thm}\label{thm: TSto on trees}
Let $G=(V,E)$ be a tree with $m=|E|$. Then, 
$$\Tchi_G^{\Gamma} = \left(q+1\right)^m\chi_G.$$
In particular, $\Tchi_G^{\Gamma}(q)$ is symmetric.
\end{thm}

\begin{proof}
We start comparing the monomial expansions of $\chi_G$ and $\Tchi^{\Gamma}_G$ in the monomial basis for symmetric and quasisymmetric functions, respectively. Let 
$\chi_G = \sum_{\lambda}c_\lambda m_\lambda$ and $\Tchi_G^{\Gamma}(q) = \sum_{\alpha}c_\alpha^o(q)M_\alpha$. Then, $\Tchi_G^{\Gamma} = \left(q+1\right)^m\chi_G$ if and only if 
$$[q^k]c_\alpha^o(q) = c_\lambda\binom{m}{k}, \qquad \text{for any $\alpha$ such that $sort(\alpha)=\lambda$.}$$

On one hand, the LHS counts the number of colorings $\kappa$ and acyclic orientations $\gamma$ such that $x^\kappa = x^\alpha$ and $\asc^\gamma(\kappa)=k$. On the other hand, in the RHS, $c_\lambda$ counts the number of colorings $\kappa'$ such that $x^\kappa = x^\beta$ for any composition $\beta$ with $sort(\beta)=\lambda$, and in particular we can take $\beta = \alpha$. Moreover, $\binom{m}{k}$ counts the number of ways of choosing $k$ edges among $m=|E(G)|$. Since any orientation in a tree is an acyclic orientation, any $k$-subset of $E(G)$ gives us an acyclic orientation so that $\asc^\gamma(\kappa)=k$. Thus, the result follows. 

\end{proof}

We finish this section with a conjecture generalizing the previous result for trees.
\begin{conj}
If $G$ is a graph for which no two cycles share an edge, then $\Tchi_G^{\Gamma}(q)$ is symmetric.
\end{conj}

\subsection{Comparison: Labelings vs acyclic orientations}\label{subsec: comparison}

These functions do share several similar properties. For instance, as shown in Proposition~\ref{prop: palindrom Label} and Theorem~\ref{thm: palindrom Acyclic}, both functions yield palindromic coefficients in their monomial expansions. Similarly, as in Theorems~\ref{thm: reverse Labeling} and~\ref{thm: reverese Acyclic}, we achieve some partial symmetry of $\Tchi_G^{\mathcal{L}}(q)$ and $\Tchi^{\Gamma}_G(q)$ in general.

One immediate difference between these two interpretations of ``normalizing" the chromatic quasisymmetric function is the behavior of the function on trees. In particular, note that if $G = \St_n^\rt$, then $\Tchi_G^{\mathcal{L}}(q)$ is not symmetric in general (e.g. Example~\ref{ex: TSt4(q)}), while Theorem~\ref{thm: TSto on trees} shows that $\Tchi_G^{\Gamma}(q)$ is always symmetric. On the contrary, $\Tchi_G^{\Gamma}$ admits a deletion-near-contraction formula that can be applied to graphs with cycles, while no such deletion-near-contraction formula has been found for $\Tchi_G^{\mathcal{L}}(q)$.

Another big difference between these two functions is in the size of the coefficients. Since there are, in general, many more potential labelings of a graph than there are acyclic orientations, $\Tchi^{\mathcal{L}}_G$ often has much larger coefficients than $\Tchi^{\Gamma}_G$, as shown in the following example.

\begin{example}
    Let $G = C_4$. Let $\Tchi_G^{\mathcal{L}}(q) = \sum_{\alpha \vdash 4}c^{\mathcal{L}}_\alpha(q)M_\alpha$ and $\Tchi^{\Gamma}_G(q) = \sum_{\alpha\vdash 4}c^{\Gamma}_\alpha(q)M_\alpha$. The following table compares $c^{\mathcal{L}}_\alpha(q)$ and $c^{\Gamma}_\alpha(q)$ for each $\alpha$ for which they are nonzero.
\begin{center}
\begin{tabular}{ |c|c|c| } 
 \hline
 $\alpha$ & $c^{\mathcal{L}}_\alpha(q)$ & $c^{\Gamma}_\alpha(q)$\\
 \hline
$(1,1,1,1)$ & $56q^4 + 128q^3 + 208q^2 + 128q + 56$ & $24q^4 + 88q^3 + 112q^2 + 88q + 24$ \\
\hline
$(1,1,2)$ & $16q^4 + 16q^3 + 32q^2 + 16q + 16$ & $4q^4 + 16q^3 + 16q^2+16q+4$\\
\hline
$(1,2,1)$ & $8q^4 + 16q^3 + 48q^2+16q+8$ & $4q^4 + 16q^3 + 16q^2+16q+4$\\
 \hline
 $(2,1,1)$ & $16q^4 + 16q^3 + 32q^2 + 16q + 16$ & $4q^4 + 16q^3 + 16q^2+16q+4$\\
 \hline
  $(2,2)$ & $8q^4 +8q^3 + 16q^2 + 8q + 8$ & $2q^4 + 8q^3 + 8q^2 + 8q + 2$\\
 \hline
\end{tabular}
\end{center}
Note also that this is another example that shows their discrepancy in terms of symmetry since $\Tchi^{\Gamma}_{C_4}(q)$ is a symmetric function, while $\Tchi_{C_4}(q)$ is not.
\end{example}

\section{The star graph}\label{sec: The star graph}
In this section, we focus our study on the star graph and examine its chromatic quasisymmetric function, as well as its total chromatic quasisymmetric functions, both with respect to labelings and orientations. Our study focuses on the expansion in the monomial bases for the corresponding framework, symmetric functions or quasisymmetric functions. The computation for the CQF and for the total orientation CQF is relatively straightforward. However, for the total labeling CQF of the star graph, the formula is more interesting and relies on proving a binomial identity, which we do in Section~\ref{sec: binomial identity}.

\subsection{The chromatic quasisymmetric function}

In this section, we include the expansion of the CQF of the star graph in the monomial quasisymmetric basis. 

Recall that we denote the star graph in $n$ vertices by $\St_n$, and the root by $\rt$. We also 
denote by $\St_n^\rt$ the star graph together with a labeling $L$ for which $L(\rt)=\rt$. Moreover, we denote by $\st_n$ the CSF of the star graph $\St_n$ and by $\st_n(q)$ its CQF. When referring to $\St_n^\rt$, we denote by $\st_n^\rt$ and by $\st_n^\rt(q)$ its CSF and its CQF, respectively. Note that $\st_n^\rt(q)$ is independent of the choice of labeling as long as  $L(\rt)=\rt$. Throughout this section, when we write $\St_n^\rt$, we assume that the root vertex is denoted by $\rt$ and that there is a labeling $L$ of the vertices of $G$ such that $L(\rt)=\rt$. 

We need to introduce some notation related to the weak compositions $\alpha$.
Let $\alpha=(\alpha_1,\ldots,\alpha_\ell)$ be the weak composition of $n$ encoding the coloring $\kappa$; that is, $\alpha_i$ counts the number of vertices with color $i$ in $\kappa$. We define $\lcomp{\alpha}{i}$ and $\rcomp{\alpha}{i}$ as the weak compositions obtained by looking at the entries strictly to the left and right of $\alpha_i$, respectively; that is, $\demph{$\lcomp{\alpha}{i}$} = (\alpha_1, \alpha_2, \dots, \alpha_{i-1})$ and $\demph{$\rcomp{\alpha}{i}$} = (\alpha_{i+1}, \alpha_{i + 2}, \dots, \alpha_\ell)$. We also denote its sizes, respectively, by $\demph{$L_{\alpha_i}$} = |\lcomp{\alpha}{i}|$ for $i>1$, with $L_{\alpha_1} = 0$, and by $\demph{$R_{\alpha_i}$} = |\rcomp{\alpha}{i}|$ for $i<\ell$, with $R_{\alpha_\ell} = 0$. Similarly, $L_{\alpha_i}^{m} = \min\{m, L_{\alpha_i}\}$ and $R_{\alpha_i}^{m} = \min\{m, R_{\alpha_i}\}$. Finally, we denote by $\mult{\alpha}$ the multinomial coefficient $\demph{$\mult{\alpha}$} = \displaystyle{\binom{|\alpha|}{\alpha_1,\ldots,\alpha_\ell}}$, where $0!=1$ and $\mult{\emptyset} = \mult{0}=1$. 

The following result summarizes all the possible values for the number of ascents of $\kappa$ in terms of the label and the coloring of the root, $L(\rt)$ and $\kappa(\rt)$. 

\begin{lemma}\label{lem: number of ascents}
Consider the star graph $\St_n^\rt$ and a proper coloring $\kappa$ such that $\kappa(\rt)=i$.  Then,
$$\asc^L(\kappa) = \left| n-\rt-L_{\alpha_i} \right| +2j, 
\qquad\quad \text{ for } \quad
0\leq j \leq \min \left\{ L_{\alpha_i}^{n-\rt},R_{\alpha_i}^{\rt-1}\right\}. $$
\end{lemma}

\begin{proof}
We start by observing that since the root $r$ has color $i$ and is connected to all the other vertices, there are no other vertices with color $i$. That is, $\alpha_i=1$. 

We split the argument into two cases depending on whether $n-\rt-L_{\alpha_i}$ is positive or negative. 
\begin{itemize}
    \item Case $L_{\alpha_i} \leq n-\rt$. In this case, $L_{\alpha_i}^{n-\rt} = L_{\alpha_i}$, and $R_{\alpha_i} = n-1-L_{\alpha_i} \geq n-1 - (n-\rt) = \rt-1$, so that $R_{\alpha_i}^{\rt-1} = \rt-1$. 
    
    Let $j$ be the number of vertices $v$ such that $L(v)<\rt$ and $\kappa(v)<i$. We want to count how many vertices $v$ there are with $L(v)>\rt$ and $\kappa(v)>i$, since these are the ones contributing to the number of ascents. We know that there are $L_{\alpha_i}-j$ vertices $v$ such that $L(v)> \rt$ and $\kappa(v)<i$. Consequently, there are $n-1-L_{\alpha_i}$ vertices with $\kappa(v)>i$, and of those, $\rt-1-j$ have $L(v)<\rt$. Therefore, there are exactly $n-1-L_{\alpha_i} - (\rt-1-j) = n-\rt-L_{\alpha_i}+j$ vertices $v$ such that $L(v)>\rt$ and $\kappa(v)>i$. Thus, $\kappa$ has $\left| n-\rt-L_{\alpha_i} \right| +2j$ ascents.

    Now, because there cannot be more than $\rt-1$ vertices with $L(v)< \rt$, we have that $j\leq \rt-1 = R_{\alpha_i}^{\rt-1}$. We also know that there are most $L_{\alpha_i}$ vertices with $\kappa(v)<i$, and so $j\leq L_{\alpha_i}^{n-\rt}$. Thus, $j\leq \min\left\{ L_{\alpha_i}^{n-\rt},R_{\alpha_i}^{\rt-1}\right\}$. 

    \item Case $L_{\alpha_i} > n-\rt$. In this case, $L_{\alpha_i}^{n-\rt} = n-\rt$, and $R_{\alpha_i}^{\rt-1} = R_{\alpha_i}$, since $R_{\alpha_i} = n-1-L_{\alpha)i} < n-1-(n-\rt) = \rt-1$. 

    Let $j$ be the number of vertices $v$ such that $L(v)>\rt$ and  $\kappa(v)>i$. We want to count how many vertices $v$ there are with $L(v)<\rt$ and $\kappa(v)<i$, since we know they also contribute to the number of ascents. We know that there are $R_{\alpha_i} - j$ vertices $v$ such that $L(v)<\rt$ and $\kappa(v)>i$.  Consequently, there are $n-1-R_{\alpha_i}$ vertices with $\kappa(v)<i$, and of those, $n-\rt-1-j$ have $L(v)>\rt$. Therefore, there are exactly $n-1-R_{\alpha_i} - (n-\rt-1-j) = \rt+R_{\alpha_i}+j$ vertices $v$ such that $L(v)>\rt$ and $\kappa(v)>i$. Recalling that $R_{\alpha_i} = n-1-L_{\alpha_i}$, we get that $\kappa$ has $\left| n-\rt-L_{\alpha_i} \right| +2j$ ascents.

    Now, because there cannot be more than $R_{\alpha_i}$ vertices with $\kappa(v)>i$, we have that $j\leq R_{\alpha_i} = R_{\alpha_i}^{\rt-1}$. We also know that there are most $n-\rt$ vertices with $L(v)>\rt$, and so $j\leq n-\rt=L_{\alpha_i}^{n-\rt}$. Thus, $j\leq \min\left\{ L_{\alpha_i}^{n-\rt},R_{\alpha_i}^{\rt-1}\right\}$. \qedhere
\end{itemize}
\end{proof}

With this lemma, we are now ready to describe the coefficients in the monomial expansion of the CQF of the star graph. 
\begin{prop}\label{prop: star QSF decomp}
Consider $\St_n^\rt$ and the $M$-expansion of its CQF,  $\displaystyle{\st_n^\rt(q) = \sum_{\alpha \models n}c_\alpha^\rt(q)M_\alpha}$. Then, 
$$
c_\alpha^\rt(q) = \sum_{i} \mult{\rcomp{\alpha}{i}}\mult{\lcomp{\alpha}{i}}  \sum_{j = 0}^{min\{L_{\alpha_i}^{n-\rt}, R_{\alpha_i}^{\rt-1}\}} \binom{n-\rt}{L_{\alpha_i}^{n-\rt}-j}\binom{\rt-1}{R_{\alpha_i}^{\rt-1}-j}q^{2j + |n-\rt-L_{\alpha_i}|},
$$
where the first summation runs over the $i$ such that $\alpha_i=1$.
\end{prop}

\begin{proof}
First, note that fixing $i$ such that $\alpha_i=1$ is essentially fixing the color of the root. That is, we want to describe the polynomial given by looking at the proper colorings $\kappa$ such that its associated weak composition is $\alpha$ and $\kappa(\rt)=i$. We count these proper colorings by first choosing those vertices that are colored less than or greater than $i$, and then choosing the specific colors afterwards.

By Lemma~\ref{lem: number of ascents}, we know that $\asc^L(\kappa) = \left| n-\rt-L_{\alpha_i} \right| +2j$, where $j$ denotes the number of vertices $v$ such that $L(v)<\rt$ and $\kappa(v)<i$, or $L(v)>\rt$ and $\kappa(v)>i$, and we also know that $0\leq j \leq \min \left\{ L_{\alpha_i}^{n-\rt},R_{\alpha_i}^{\rt-1}\right\}$. The argument now splits into two cases, one where $n - \rt \geq L_{\alpha_i}$ and one where $n - \rt < L_{\alpha_i}$. We will show the first case, but the second case follows a similar argument.

Suppose that $n - \rt \geq L_{\alpha_i}$. Following a similar process to the one outlined in the proof of Lemma~\ref{lem: number of ascents}, we choose $j$ vertices such that $L(v) < \rt$ to be colored such that $\kappa(v) < i$. There are $\rt-1$ of these vertices, so there are $\binom{\rt-1}{j}$ ways to do this. Note that since $\rt-1 = R_{\alpha_i}^{\rt-1}$ in this case, $\binom{\rt-1}{j} = \binom{\rt-1}{R_{\alpha_i}^{\rt-1}-j}$. We still need to choose the remaining $L_{\alpha_i} - j$ vertices that must be colored less than $i$. Since $L_{\alpha_i} = L_{\alpha_i}^{n-\rt}$ in this case, we choose $L_{\alpha_i}^{n-\rt} - j$ vertices from the $n - \rt$ vertices such that $L(v) > \rt$ and there are $\binom{n-\rt}{L_{\alpha_i}^{n-\rt}-j}$ ways to do it. 

Thus, we have determined exactly which vertices are colored greater than $i$ and which are colored less than $i$. Finally, the factor $\displaystyle{\mult{\rcomp{\alpha}{i}}\mult{\lcomp{\alpha}{i}}}$ comes from choosing the precise color for every vertex. This gives, in total, $\mult{\rcomp{\alpha}{i}}\mult{\lcomp{\alpha}{i}}\binom{n-\rt}{L_{\alpha_i}^{n-\rt}-j}\binom{\rt-1}{R_{\alpha_i}^{\rt-1}-j}$ ways to choose one of these proper colorings, and this proves the claim.
\end{proof}

We also have the following property, which will be useful later in the paper.

\begin{lemma}\label{equality of alpha^r} 
Let $\St^\rt_n$ be the star graph on $n$ vertices, and consider the $M$-expansion of its CQF, $\st_n^\rt (q) = \sum_{\alpha} c^\rt_\alpha(q)M_\alpha$. Then, 
$$
[q^k]c_\alpha^\rt(q) = [q^{n-1-k}]c_\alpha^{n-1-\rt}(q) \hspace{0.5cm} \text{and} \hspace{0.5cm} [q^k]c_\alpha^\rt(q) = [q^{n-1-k}]c_{\alpha^{\rev}}^\rt,
$$
where $\alpha^{\rev}$ is the reverse of $\alpha$.
\end{lemma}

\begin{proof}
We first prove that $[q^k]c_\alpha^\rt(q) = [q^{n-1-k}]c_\alpha^{n-1-\rt}(q)$. Let $L$ be a labeling of $\St_n^\rt$ such that $L(\rt)=\rt$, and define $L^{\rev}$ as the labeling defined by $L^{\rev}(v) = n-1-i$ if $L(v)=i$. Then,  the root is labeled $L(\rt)=n-1-\rt$ and $L^{\rev}$ is a labeling of $\St_n^{n-1-\rt}$.
Now, $[q^k]c_\alpha^\rt(q)$ counts the number of colorings $\kappa$ of $\St_n^\rt$ with associated composition $\alpha$ and that has $k$ ascents with respect to $L$. These are exactly the colorings of $\St_n^{n-1-\rt}$ with associated composition $\alpha$ and that have $n-1-k$ ascents with respect to $L^{\rev}$. Thus, $[q^k]c_\alpha^\rt(q) = [q^{n-1-k}]c_\alpha^{n-1-\rt}(q)$.

Next, we prove that $[q^k]c_\alpha^\rt(q) = [q^{n-1-k}]c_{\alpha^{\rev}}^\rt$ similarly. Let $L$ be a labeling of $\St_n^\rt$ such that $L(\rt)=\rt$, and consider a coloring $\kappa$ of $\St_n^\rt$ with associated composition $\alpha$ and that has $k$ ascents with respect to $L$. We define the coloring $\kappa^{\rev}$ by setting $\kappa^{\rev}(v) = n-1-i$ if $\kappa(v)=i$. Therefore, $\kappa^{\rev}$ is a coloring of $\St^\rt_n$ with associated composition $\alpha^{\rev}$ and that has $n-1-k$ ascents. Thus, $[q^k]c_\alpha^\rt(q) = [q^{n-1-k}]c_{\alpha^{\rev}}^\rt$.
 \end{proof}

\subsection{Total CQF via acyclic orientations}

In this section, we compute the expansion of the total orientation CQF of $\St_n$, which is a direct consequence of the $m$-expansion of the CSF of $\St_n$. Before stating the results, we introduce some notation. Given a partition $\lambda = (\lambda_1, \lambda_2, \dots, \lambda_\ell)$, we can rewrite it in terms of its repeated parts, $\lambda = \left\langle 1^{a_1}\ldots k^{a_k}\right\rangle$, where $a_i$ denotes the number of entries in $\lambda$ equal to $i$. We write $a_i(\lambda)$ whenever we need to specify the partition $\lambda$. We also denote $\demph{$\tilde{\lambda}$} = (\lambda_1, \lambda_2, \dots, \lambda_{\ell - 1})$.

\begin{thm}\label{thm: starCSF}
Consider $\St_n$ and the $m$-expansion of its CSF,   
$\st_n = \sum_{\lambda \vdash n} c_\lambda m_\lambda$. Then, $$c_\lambda = a_1(\lambda)\cdot \mult{\tilde{\lambda}}.$$
\end{thm}

\begin{proof}
Recall that $c_\lambda$ counts the number of proper colorings of $\St_n$ with associated composition $\lambda$. We can count these colorings by choosing first a color for the root, and then choosing colors for all other vertices. We start noticing that $a_1(\lambda) \geq 1$ since for any proper coloring of $\St_n$, the color of the root is never repeated. Now, the root $r$ can be colored any color $i$ such that $\lambda_i=1$, and so there are $a_1(\lambda)$ ways to color the root. Suppose that the root is colored $i$. Then, there are $\mult{(\lambda_1, \lambda_2, \ldots, \lambda_{i-1}, \lambda_{i+1},\ldots, \lambda_\ell)}$ ways to color the remaining vertices. Since $\lambda$ is a partition, we have that $\lambda_i=\lambda_{i+1}=\ldots = \lambda_{\ell} =1$, and so $\mult{(\lambda_1, \lambda_2, \ldots, \lambda_{i-1}, \lambda_{i+1},\ldots, \lambda_\ell)}=\mult{\tilde{\lambda}}$ and the result follows. 
\end{proof}

Applying Theorems~\ref{thm: TSto on trees} and~\ref{thm: starCSF}, we obtain the following.
\begin{cor}
    Let $G = \St_n$. Then, 
    $$\Tchi_{\St_n}^o(q) = \sum_{\lambda \vdash n}a_1(\lambda)\cdot \mult{\tilde{\lambda}}(q+1)^{n-1}m_\lambda.$$
\end{cor}

\subsection{Total CQF via labelings}\label{subsec: total star labelings}

As mentioned before, the CQF of the star graph essentially depends on the color of the root, which allows us to focus on the following normalization for the total labeling CQF. 
\begin{defn}\label{def: Tstar}
We define the \demph{normalized total labeling CQF of $\St_n$} as 
$$\Tst_n(q) := \sum_{\text{labelings of the root}} \chi_{\st_n^\rt}(q) = \frac{1}{(n-1)!}\Tchi^{\mathcal{L}}_{\St_n}(q).$$
\end{defn}

We want to analyze the $M$-expansion of $\Tst_n(q)$, and we start by looking at an example. Note that in this section, we denote the coefficients with the letter $b$, so that it is more explicit that these are not the same as the coefficients for the total labeling CQF of the star graph. 
\begin{example}\label{ex: TSt4(q)}
Consider $\St_n$ and consider the $M$-expansion of $\Tst_4(q)=\sum_\alpha b_\alpha(q) M_\alpha$. In Table~\ref{tab: calpha for St_4}, we show the coefficients $b_\alpha(q)$.

\begin{table}[H]
\centering
\begin{tabular}{ |c|c| } 
 \hline
 $\alpha$  & $b_\alpha(q)$\\ \hline
 $[1,1,1,1]$ & $16q^3 + 32q^2 + 32q + 16$\\ 
 $[1,1,2]$ & $4q^3 + 8q^2 + 8q + 4$\\ 
 $[1,2,1]$ & $6q^3 + 6q^2 + 6q + 6$\\
 $[1,3]$ & $q^3 + q^2 + q + 1$\\
 $[2,1,1]$ & $4q^3 + 8q^2 + 8q + 4$\\
 $[3,1]$ & $q^3 + q^2 + q + 1$\\
 \hline
\end{tabular}
\caption{Coefficients in the $M$-expansion of $\Tst_4(q)$.}
    \label{tab: calpha for St_4}
\end{table}

For comparison, in Table~\ref{tab: calphak for St_4}, we show the coefficients $c_\alpha^\rt(q)$ in the $M$-expansion of $\st_n^\rt(q)$, where we recall that $\rt$ corresponds to the color of the root.
\begin{table}[H]
    \centering
\begin{tabular}{ |c|c|c|c|c| } 
 \hline
 $\alpha$ & $c_\alpha^1(q)$ &  $c_\alpha^2(q)$ & $c_\alpha^3(q)$ &  $c_\alpha^4(q)$ \\ \hline
 $[1,1,1,1]$ & $6q^3 + 6q^2 + 6q + 6$ & $2q^3 + 10q^2 + 10q + 2$ & $2q^3 + 10q^2 + 10q + 2$ & $6q^3 + 6q^2 + 6q + 6$\\ 
 $[1,1,2]$ & $3q^3 + 3q^2$ & $q^3 + 3q^2 + 2q$ & $2q^2 + 3q + 1$ & $3q + 3$ \\ 
 $[1,2,1]$ & $3q^3 + 3$ & $3q^2 + 3q$ & $3q^2 + 3q$ & $3q^3 + 3$ \\
 $[1,3]$ & $q^3$ & $q^2$ & $q$ & $1$ \\
 $[2,1,1]$ & $3q + 3$ & $2q^2 + 3q + 1$ & $q^3 + 3q^2 + 2q$ & $3q^3 + 3q^2$ \\
 $[3,1]$ & 1 & $q$ & $q^2$ & $q^3$\\
 \hline
\end{tabular}
    \caption{Coefficients in the $M$-expansion of $\st^\rt_4(q)$ for $1\leq \rt \leq 4$.}
    \label{tab: calphak for St_4}
\end{table}

\end{example}

Before continuing, we introduce more notation associated with $\St_n^\rt$ and the normalized total labeling CQF, $\Tst_n(q)$. Firstly, for each possible labeling of the root, $1\leq \rt \leq \lfloor n/2 \rfloor$, we fix a \demph{representative labeling}  such that the label of the root is $\rt$, and we denote that labeling by $L_\rt$. For the other possible values of the labeling of the root, $\lfloor n/2 \rfloor< \rt' \leq n$, we write $\rt' = n+1-\rt$ and define $L_{\rt'}$ as the labeling $L_{\rt'}(v) = n+1-L_\rt(v)$, for any $v\in V(\St_n)$. We refer to $L_{\rt'}$ as the \demph{complementary labeling}. We denote by $\mathcal{K}_\alpha^\rt$ the set of colorings of $\St_n^\rt$ with associated composition $\alpha$ and by $\mathcal{K}_\alpha^\rt(i)$ the subset of colorings of $\mathcal{K}_\alpha^\rt$ such that $\kappa(\rt)=i$.
Now, consider the $M$-expansion of $\Tst_n(q)$, $\Tst_n(q)= \sum_\alpha b_\alpha(q)M_\alpha$. Then, we have that
$$
b_\alpha(q) = \sum_{\rt=1}^n c_\alpha^\rt(q) = \sum_i b_\alpha(i;q) = \sum_{i,\rt} b_\alpha^\rt(i;q),
$$
where the first sum runs over the possible labelings of the root and we recall that $c_\alpha^\rt(q)$ is independent of the concrete labeling $L$ as long as $L(\rt) = \rt$; the second sum runs over the $i$ such that $\alpha_i=1$, which corresponds to the color of the root $\kappa(\rt)$; and the third sum runs over the possible labelings and colorings of the root. That is, 
$$
c_\alpha^\rt(q)  = \sum_{\kappa \in \mathcal{K}^\rt_\alpha} q^{\asc^L(\kappa)} 
\qquad \text{ and } \qquad
b_\alpha^\rt(i;q) = \sum_{\kappa \in \mathcal{K}^\rt_\alpha(i)} q^{\asc^L(\kappa)}. 
$$

\begin{lemma}\label{lem: symmetry of B-coeff}
Let $\St_n$ be the star graph in $n$ vertices, and consider the $M$-expansion of its normalized total labeling CQF $\Tst_n(q)= \sum_\alpha b_\alpha(q)M_\alpha$, where $b_\alpha(q)= \sum_{i,\rt} b^\rt_\alpha(i;q)$. Then, 
$$
[q^k]b^\rt_\alpha(i;q) = [q^k]b^{n+1-\rt}_{\alpha^{\rev}}(\ell-i+1;q), \qquad \text{ where } \ell = \ell(\alpha).
$$
That is, $b_\alpha(i;q) = b_{\alpha^{\rev}}(\ell-i+1;q)$.

\end{lemma}

\begin{proof}
Note that $[q^k]b^\rt_\alpha(i;q)$ counts the number of colorings $\kappa \in \mathcal{K}^\rt_\alpha(i)$ such that $\asc^L(\kappa)=k$ and that $[q^k]b^{n+1-\rt}_{\alpha^{\rev}}(\ell-i+1;q)$ counts the number of colorings $\kappa' \in \mathcal{K}^{n+1-\rt}_{\alpha^{\rev}}(\ell-i+1)$ such that $\asc^L(\kappa')=k$. 

Consider the following map:
$$
\begin{array}{cccc} 
\varphi: & \left\{ \kappa \in \mathcal{K}^\rt_\alpha(i)\ |\ \asc^L(\kappa)=k\right\} 
& \longrightarrow & 
\left\{ \kappa' \in \mathcal{K}^{n+1-\rt}_{\alpha^{\rev}}(\ell+1-i)\ |\ \asc^L(\kappa')=k\right\} \\
& \kappa &\longmapsto & \varphi(\kappa):= \kappa',
\end{array}$$ 
where $\kappa'$ is the coloring $\kappa'(v) = \ell - \kappa(v)+1$, for all $v\in \St_n$.

By definition of $\varphi$, $\kappa'$ is a coloring in $\mathcal{K}^{n+1-\rt}_{\alpha^{\rev}}(\ell+1-i)$. 
Moreover, it is easy to see that taking the representative labeling and its complementary labeling, the edges that yield an ascent in the domain of $\varphi$ are exactly the edges that yield an ascent in its image, and so the number of ascents is preserved and $\varphi$ is a well-defined map. The fact that $\varphi$ is a bijection follows mostly by definition. 
\end{proof}

Now we are ready to state the formula for the $M$-expansion of the normalized total labeling CQF of $\St_n$.
\begin{thm}
Let $\St_n$ be the star graph on $n$ vertices, and consider the $M$-expansion of its normalized total labeling CQF, $\Tst_n(q) = \sum_{\alpha}b_\alpha(q)M_\alpha$. Then, 
$$b_\alpha(q) = \sum_{i} \mult{\rcomp{\alpha}{i}}\mult{\lcomp{\alpha}{i}} \sum_{l=0}^{s}\binom{n}{l}q^l[n-2l]_q,$$
where the first sum runs over all the $i$ such that $\alpha_i=1$, and $s=  \min\{R_{\alpha_i},L_{\alpha_i}\}$.
Equivalently, 
$$[q^k]b_\alpha(i;q) = \mult{\rcomp{\alpha}{i}}\mult{\lcomp{\alpha}{i}}\sum_{l=0}^{s_0}\binom{n}{l},$$
where $s_0 = \min\{k,s\}$ and $0\leq k\leq n-1$.
\end{thm}

This result is a consequence of the two results we state below. The first result is another interpretation of the coefficients $b_\alpha(q)$, which is proved in Section~\ref{subs: proof of 1st step}. 

\begin{thm}\label{thm: first step c_alpha}
Let $\St_n$ be the star graph on $n$ vertices, and consider the $M$-expansion of its normalized total labeling CQF, $\Tst_n(q) = \sum_{\alpha}b_\alpha(q)M_\alpha$.
For $k \leq \lfloor \frac{n-1}{2}\rfloor$, 
$$
[q^k]b_\alpha(q) = \sum_{i}
\mult{\rcomp{\alpha}{i}}\mult{\lcomp{\alpha}{i}}
\sum_{j=0}^{s_0}\binom{s + k -2j}{s-j}\binom{n-1-s - k +2j}{j},
$$
where the first sum runs over all the $i$ such that $\alpha_i=1$, $s = \min\{R_{\alpha_i}, L_{\alpha_i}\}$, and $s_0 = \min\{s,k\}$. 
Equivalently, 
$$[q^k]b_\alpha(i;q) = \mult{\rcomp{\alpha}{i}}\mult{\lcomp{\alpha}{i}} \sum_{j=0}^{s_0}\binom{s + k -2j}{s-j}\binom{n-1-s - k +2j}{j}.$$
\end{thm}
The second result is a binomial identity, which is more challenging to prove than we expected. In Section~\ref{sec: binomial identity} we present a combinatorial model for it together with its proof. 
\begin{thm}\label{prop: binomial identity}
Let $n$, $k$, and $s$ be non-negative integers with $k,s\leq \left\lfloor \frac{n-1}{2} \right\rfloor$, and $s_0 = \min\{s,k\}$. Then,
$$
\sum_{j=0}^{s_0} \binom{s + k -2j}{s-j}\binom{n-1-s - k +2j}{n-1-s-k +j} = \sum_{l = 0}^{s_0}\binom{n}{l}.
$$
\end{thm}

\begin{remark} 
Note that by Proposition~\ref{prop: palindrom Label}, we know that the coefficients $b_\alpha(q)$ are symmetric, and so in Theorems~\ref{thm: first step c_alpha} and~\ref{prop: binomial identity} we only need to consider $0\leq k \leq \left\lfloor \frac{n-1}{2} \right\rfloor$. Moreover, $s$ is exactly the minimum between the number of entries to the left and right of the entry $\alpha_i=1$ that we consider, and so, in Theorem~\ref{prop: binomial identity}, we also have that $s\leq \left\lfloor \frac{n-1}{2} \right\rfloor$. 
\end{remark}

\subsection{Proof of Theorem~\ref{thm: first step c_alpha}}\label{subs: proof of 1st step}

\begin{proof}
By Definition~\ref{def: Tstar} and Proposition~\ref{prop: star QSF decomp}, we have that
\begin{align} 
b_\alpha(q) &= \sum_{\rt=1}^n c_\alpha^\rt(q) \nonumber \\
            &= \sum_{\rt=1}^n \sum_{i : \alpha_i = 1}\sum_{p = 0}^{min\{L_{\alpha_i}^{n-\rt}, R_{\alpha_i}^{\rt-1}\}} \mult{\rcomp{\alpha}{i}}\mult{\lcomp{\alpha}{i}}\binom{n-\rt}{L_{\alpha_i}^{n-\rt}-p}\binom{\rt-1}{R_{\alpha_i}^{\rt-1}-p}q^{2p + |n-\rt-L_{\alpha_i}|} \nonumber \\
            & = \sum_{i : \alpha_i = 1}\sum_{\rt=1}^n \sum_{p = 0}^{min\{L_{\alpha_i}^{n-\rt}, R_{\alpha_i}^{\rt-1}\}} \mult{\rcomp{\alpha}{i}}\mult{\lcomp{\alpha}{i}}\binom{n-\rt}{L_{\alpha_i}^{n-\rt}-p}\binom{j-1}{R_{\alpha_i}^{\rt-1}-p}q^{2p + |n-\rt-L_{\alpha_i}|}. \label{eq: old decomp in proof 3.8} 
\end{align}

Let us fix $i$ such that $\alpha_i=1$, that is, let us fix the color of the root $\kappa(\rt)=i$, for $1\leq r \leq n$. 
By Lemma~\ref{lem: symmetry of B-coeff}, we can assume that $L_{\alpha_i}\leq R_{\alpha_i}$, so that $s=\min\{R_{\alpha_i}, L_{\alpha_i}\} = L_{\alpha_i}$.

Suppose we fix $k \leq \lfloor \frac{n-1}{2}\rfloor$. Then, we have that $k=2p + |n-r-L_{\alpha_i}|$ for the different values of $r$ and $p$. We claim that the smallest labeling of the root of $\St_n$, that is $r$, that contributes a nonzero amount to the coefficient of $q^k$ corresponds to taking $p=0$ and $\rt_0=R_{\alpha_i}-k+1$. 
To see this, we first assume that $p=0$. Then, recall that $n = R_{\alpha_i}+L_{\alpha_i}+1$, and so $k = |R_{\alpha_i}+L_{\alpha_i}+1-\rt-L_{\alpha_i}| = |R_{\alpha_i}+1-\rt|$. Thus, $\rt_0 = \min_\rt \{ R_{\alpha_i}-k+1, R_{\alpha_i}+k-1\} = R_{\alpha_i}-k+1$ as desired. Now, assume that $p>0$ and suppose that there exists some $r_1 = \rt_0 -\varepsilon$, with $\varepsilon>0$. Then, we have that $k=2p + |n-r-L_{\alpha_i}|= k=2p + |R_{\alpha_i} - r_1 +1| = 2p + |k +\varepsilon| > k$, which is a contradiction.  

In fact, we can write $\rt=\rt_0+2j$, with $0 \leq j \leq \lfloor \frac{n-\rt_0}{2}\rfloor$.
The parity of $k=2p + |R_{\alpha_i}+1-\rt|$ depends on the parity of $\rt$, and so the only nonzero contributions correspond to the labeling of the form $\rt = \rt_0 +2j$, with $0\leq j \leq \lfloor \frac{n-\rt_0}{2}\rfloor$.

Therefore, in Equation~\eqref{eq: old decomp in proof 3.8}, we have that the coefficient of $q^k$ in $b_\alpha(q)$ is
\begin{align} 
[q^k] \left(\sum_{i : \alpha_i = 1}\sum_{j=0}^{\lfloor \frac{n-\rt_0}{2}\rfloor} \sum_{p = 0}^{min\{L_{\alpha_i}^{n-\rt_0-2j}, R_{\alpha_i}^{\rt_0+2j-1}\}} \mult{\rcomp{\alpha}{i}}\mult{\lcomp{\alpha}{i}}\binom{n-\rt_0-2j}{L_{\alpha_i}^{n-\rt_0-2j}-p}\binom{\rt_0+2j-1}{R_{\alpha_i}^{\rt_0+2j-1}-p}q^{2p + |k-2j|} \right),\label{eq: coeff with l summation} 
\end{align}
where in the exponent of $q$ we use that $|n-\rt_0-2j-L_{\alpha_i}| = |R_{\alpha_i}-\rt_0-2j+1| = |k-2j|$.

Our next step is to calculate the coefficient for the labels $\rt_0+2j$ of the root. Now, we first want to show that this contribution is nonzero exactly when $0 \leq j \leq s_0 = \min \{s, k\}$. Note that the condition $j \leq k$ comes from the exponent of $q$ while the condition $j \leq s$ comes from the binomial coefficients. 

Suppose that $j > k$. Then $|k - 2j| > k$ and so, the exponent of $q$ is $2p + |k - 2j|>k$. That is, it does not contribute to the coefficient of $q^k$ that we are looking at. Thus, we have to $j\leq k$. 
Moreover, note that if $j$ is fixed, then $2p + |k-2j| = k$ has only one solution in terms of $p$. In fact, $p$ takes the following values:
\begin{itemize}
    \item[(a)] If $2j \leq k$, then $p=j$. This follows since $2p + |k-2j| = 2p+k-2j$, and so $2p+k-2j = k$ implies that $p=j$.

    \item[(b)] If $2j \geq k$, then $p=k-j$. This follows since $2p + |k-2j| = 2p-k+2j$, and so $2p-k+2j = k$ implies that $p=k-j$.
\end{itemize}

Now we want to show the following claim:
$$p\leq \min\{L_{\alpha_i}^{n-\rt_0-2j},R_{\alpha_i}^{\rt_0+2j-1}\} \qquad \text{ if and only if } \qquad j\leq s .$$
Firstly, we analyze $\min\{L_{\alpha_i}^{n-\rt_0-2j},R_{\alpha_i}^{\rt_0+2j-1}\}$ depending on the values of $p$. Note that, since $L_{\alpha_i}\leq R_{\alpha_i}$ and $L_{\alpha_i}= n - \rt_0 - k$, we have that 
$$\min\{L_{\alpha_i}^{n-\rt_0-2j},R_{\alpha_i}^{\rt_0+2j-1}\} = 
\min\{n-\rt_0-2j, n - \rt_0 - k, \rt_0+2j-1\}.$$
Then, 
\begin{enumerate}
    \item[(1)] $\min\{L_{\alpha_i}^{n-\rt_0-2j},R_{\alpha_i}^{\rt_0+2j-1} \}= n-\rt_0-2j$ implies that $p=k-j$;
    \item[(2)] $\min\{L_{\alpha_i}^{n-\rt_0-2j},R_{\alpha_i}^{\rt_0+2j-1}\}  = n - \rt_0 - k$ implies that $p=j$; and 
    \item[(3)] $\min\{L_{\alpha_i}^{n-\rt_0-2j},R_{\alpha_i}^{\rt_0+2j-1}\} = \rt_0+2j-1$ implies that $p=j$; 
\end{enumerate}
Assume $p\leq \min\{L_{\alpha_i}^{n-\rt_0-2j},R_{\alpha_i}^{\rt_0+2j-1}\}$. Then, analyzing the possible cases above, it is straightforward to conclude that $j\leq s$. For the other direction, we assume that $j\leq s$ and use the values of $p$ in cases $(a)$ and $(b)$ to conclude that $p\leq \min\{L_{\alpha_i}^{n-\rt_0-2j},R_{\alpha_i}^{\rt_0+2j-1}\}$.

Thus, continuing from Equation~\eqref{eq: coeff with l summation}, we have that the coefficient of $q^k$ in $b_\alpha(q)$ is
\begin{align} 
[q^k] \left(\sum_{i : \alpha_i = 1}\sum_{j=0}^{s_0} \sum_{p = 0}^{min\{L_{\alpha_i}^{n-\rt_0-2j}, R_{\alpha_i}^{\rt_0+2j-1}\}} \mult{\rcomp{\alpha}{i}}\mult{\lcomp{\alpha}{i}}\binom{n-\rt_0-2l}{L_{\alpha_i}^{n-\rt_0-2j}-p}\binom{\rt_0+2j-1}{R_{\alpha_i}^{\rt_0+2j-1}-p}q^{2p + |k-2j|} \right), \label{eq: coeff with final l} 
\end{align}

Finally, we are ready to look at the coefficients and rewrite them in terms of our parameters. First, we notice the following:
\begin{align*}
    L_{\alpha_i}^{n-\rt_0-2j} &= \min \left\{ L_{\alpha_i}, n-\rt_0-2j\right\} = \min \left\{ L_{\alpha_i}, L_{\alpha_i}-k+2j \right\} = \min \left\{ s, s-k+2j \right\} , \\
    R_{\alpha_i}^{\rt_0+2j-1} &= \min \left\{ R_{\alpha_i}, \rt_0+2j-1 \right\} = \min \left\{ n-s-1, n-s-1-k+2j  \right\},
\end{align*}
where we use that $\rt_0 = R_{\alpha_i}-k+1$, $L_{\alpha_i} +k = n-\rt_0$, $L_{\alpha_i}+R_{\alpha_i}+1=n$, and our assumption that $L_{\alpha_i}\leq R_{\alpha_i}$ and so $s=\min \{L_{\alpha_i},R_{\alpha_i}\}=L_{\alpha_i}$. 
Then, following the possible values of $p$, we have that 
\begin{itemize}
    \item[(a)] If $2j \leq k$, then $p=j$, and we have that 
    $$
    \binom{n-\rt_0-2j}{L_{\alpha_i}^{n-\rt_0-2j}-p}\binom{\rt_0+2j-1}{R_{\alpha_i}^{\rt_0+2j-1}-p} = 
    \binom{s+k-2j}{s-j}\binom{n-s-k+2j-1}{n-s-1-k+2j-j}.
    $$
    \item[(b)] If $2j \geq k$, then $p=k-j$, and we have that 
    $$
    \binom{n-\rt_0-2j}{L_{\alpha_i}^{n-\rt_0-2j}-p}\binom{\rt_0+2j-1}{R_{\alpha_i}^{\rt_0+2j-1}-p} = 
    \binom{s+k-2j}{s+k-2j-(k-j)}\binom{n-s-k+2j-1}{n-s-1-(k-j)}.
    $$
\end{itemize}
In both cases, we have the simplification
$$
\binom{n-\rt_0-2j}{L_{\alpha_i}^{n-\rt_0-2j}-p}\binom{\rt_0+2j-1}{R_{\alpha_i}^{\rt_0+2j-1}-p} = \binom{s+k-2j}{s-j}\binom{n-s-k+2j-1}{j},
$$
and the result follows. 
\end{proof}

\section{The binomial identity and its combinatorial model}\label{sec: binomial identity}

In this section, we present a combinatorial proof of the following binomial identity.
\begin{thm}\label{thm:binomial_identity}
    Let $n$, $k$, and $s$ be nonnegative integers with $s\leq  k \leq \left\lfloor \frac{n-1}{2} \right\rfloor$. Then,
    \begin{equation}\label{eq:binomial_identity}
         \sum_{j=0}^{s} \binom{s + k -2j}{s-j}\binom{n-1-s - k +2j}{j} = \sum_{l = 0}^{s}\binom{n}{l}.
    \end{equation}
\end{thm}
\begin{remark}
    Theorem~\ref{thm:binomial_identity} is slightly different than the statement in Theorem~\ref{prop: binomial identity} since we assume that $s\leq k \leq \left\lfloor \frac{n-1}{2} \right\rfloor$. Now, looking at the formula in Theorem~\ref{prop: binomial identity}, we note that the LHS is symmetric in $s$ and $k$, and the RHS does not depend on $s$ and $k$ other than the minimum $s_0=\min\{s,k\}$. Thus, we can assume that $s\leq k$, and so $s_0=s$. 
\end{remark}

In the following example, we observe that the terms on each side do not match directly. Also, in Equation~\eqref{eq:binomial_identity}, the LHS depends on $k$ while the RHS does not. 
\begin{example}\label{ex:running example}
    Let $s = k = 2$ and $n = 6$. Then, the LHS of Equation~\eqref{eq:binomial_identity} gives
    \begin{align*}
        \binom{4}{2}\binom{1}{1} + \binom{2}{1}\binom{3}{2} + \binom{0}{0}\binom{5}{3} = 6 + 6 + 10 = 22
    \end{align*}
    while the RHS gives 
    \begin{align*}
        \binom{6}{0} + \binom{6}{1} + \binom{6}{2} = 1 + 6 + 15 = 22.
    \end{align*}
\end{example}
We prove this binomial identity combinatorially using a model for which the computation on the LHS of Equation~\eqref{eq:binomial_identity} is direct, and we work our way to the RHS. 

\subsection{Combinatorial model}

Consider a row of $n$ boxes. We say that a box decorated with
\scalebox{0.35}{\begin{tikzpicture}
   \fill[DarkBlue!30] (0,0) rectangle (1,1);
   \draw[step=1cm,black,very thin] (0,0) grid (1,1);
\end{tikzpicture}} 
is a \demph{barrier} and that a box decorated with 
\scalebox{0.35}{\begin{tikzpicture}
   \draw[step=1cm,black,very thin] (0,0) grid (1,1);
   \node at (0.5,0.5) {\scalebox{3}{$\times$}};
\end{tikzpicture}} 
is a \demph{marked box}. We define a \demph{configuration} as a row of $n$ boxes with $s$ marked boxes. We encode this information as an increasing sequence of positive integers, $\gamma = (\gamma_1,\ldots,\gamma_s)$, where $\gamma_i\in [n]$ indicates the position of the $i^{\text{th}}$ marked box labeled from left to right. We say that a configuration satisfies the \demph{$j$-condition} if the box at position $s+k+1-2j$ can be decorated as a barrier and the configuration contains $s-j$ marked boxes to the left of the barrier, no marked boxes at the barrier, and $j$ marked boxes to the right of the barrier. 

\begin{example}
  In Figure~\ref{fig:all_configurations}, we include all 15 configurations that we are considering.
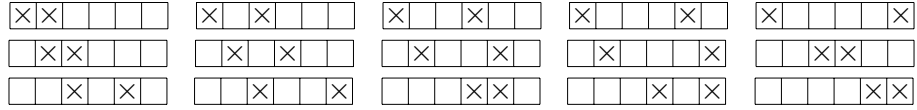
\begin{figure}[ht]
\centering
$\begin{array}{ccccc}
\scalebox{0.35}{\begin{tikzpicture}
   \draw[step=1cm,black,very thin] (0,0) grid (6,1);
   \node at (0.5,0.5) {\scalebox{3}{$\times$}};
   \node at (1.5,0.5) {\scalebox{3}{$\times$}};
\end{tikzpicture}} & 
\scalebox{0.35}{\begin{tikzpicture}
   \draw[step=1cm,black,very thin] (0,0) grid (6,1);
   \node at (0.5,0.5) {\scalebox{3}{$\times$}};
   \node at (2.5,0.5) {\scalebox{3}{$\times$}};
\end{tikzpicture}} &
\scalebox{0.35}{\begin{tikzpicture}
   \draw[step=1cm,black,very thin] (0,0) grid (6,1);
   \node at (0.5,0.5) {\scalebox{3}{$\times$}};
   \node at (3.5,0.5) {\scalebox{3}{$\times$}};
\end{tikzpicture}} &
\scalebox{0.35}{\begin{tikzpicture}
   \draw[step=1cm,black,very thin] (0,0) grid (6,1);
   \node at (0.5,0.5) {\scalebox{3}{$\times$}};
   \node at (4.5,0.5) {\scalebox{3}{$\times$}};
\end{tikzpicture}} &
\scalebox{0.35}{\begin{tikzpicture}
   \draw[step=1cm,black,very thin] (0,0) grid (6,1);
   \node at (0.5,0.5) {\scalebox{3}{$\times$}};
   \node at (5.5,0.5) {\scalebox{3}{$\times$}};
\end{tikzpicture}} \\
\scalebox{0.35}{\begin{tikzpicture}
   \draw[step=1cm,black,very thin] (0,0) grid (6,1);
   \node at (1.5,0.5) {\scalebox{3}{$\times$}};
   \node at (2.5,0.5) {\scalebox{3}{$\times$}};
\end{tikzpicture}} & 
\scalebox{0.35}{\begin{tikzpicture}
   \draw[step=1cm,black,very thin] (0,0) grid (6,1);
   \node at (1.5,0.5) {\scalebox{3}{$\times$}};
   \node at (3.5,0.5) {\scalebox{3}{$\times$}};
\end{tikzpicture}} &
\scalebox{0.35}{\begin{tikzpicture}
   \draw[step=1cm,black,very thin] (0,0) grid (6,1);
   \node at (1.5,0.5) {\scalebox{3}{$\times$}};
   \node at (4.5,0.5) {\scalebox{3}{$\times$}};
\end{tikzpicture}} & 
\scalebox{0.35}{\begin{tikzpicture}
   \draw[step=1cm,black,very thin] (0,0) grid (6,1);
   \node at (1.5,0.5) {\scalebox{3}{$\times$}};
   \node at (5.5,0.5) {\scalebox{3}{$\times$}};
\end{tikzpicture}} &
\scalebox{0.35}{\begin{tikzpicture}
   \draw[step=1cm,black,very thin] (0,0) grid (6,1);
   \node at (2.5,0.5) {\scalebox{3}{$\times$}};
   \node at (3.5,0.5) {\scalebox{3}{$\times$}};
\end{tikzpicture}} \\
\scalebox{0.35}{\begin{tikzpicture}
   \draw[step=1cm,black,very thin] (0,0) grid (6,1);
   \node at (2.5,0.5) {\scalebox{3}{$\times$}};
   \node at (4.5,0.5) {\scalebox{3}{$\times$}};
\end{tikzpicture}} &
\scalebox{0.35}{\begin{tikzpicture}
   \draw[step=1cm,black,very thin] (0,0) grid (6,1);
   \node at (2.5,0.5) {\scalebox{3}{$\times$}};
   \node at (5.5,0.5) {\scalebox{3}{$\times$}};
\end{tikzpicture}} &
\scalebox{0.35}{\begin{tikzpicture}
   \draw[step=1cm,black,very thin] (0,0) grid (6,1);
   \node at (3.5,0.5) {\scalebox{3}{$\times$}};
   \node at (4.5,0.5) {\scalebox{3}{$\times$}};
\end{tikzpicture}} &
\scalebox{0.35}{\begin{tikzpicture}
   \draw[step=1cm,black,very thin] (0,0) grid (6,1);
   \node at (3.5,0.5) {\scalebox{3}{$\times$}};
   \node at (5.5,0.5) {\scalebox{3}{$\times$}};
\end{tikzpicture}} & 
\scalebox{0.35}{\begin{tikzpicture}
   \draw[step=1cm,black,very thin] (0,0) grid (6,1);
   \node at (4.5,0.5) {\scalebox{3}{$\times$}};
   \node at (5.5,0.5) {\scalebox{3}{$\times$}};
\end{tikzpicture}}
\end{array}$
    \caption{All the configurations for $n=6$ and $s=k=2$.}
    \label{fig:all_configurations}
\end{figure}  
\end{example}

Now, we fix $n$ and $s$, and we look at the configurations of a row of $n$ boxes with $s$ marked ones. For $0\leq j\leq s$, we denote by \demph{$B_j$}  the number of these configurations that satisfy the $j$-condition. For $1\leq l \leq s+1$, we denote by \demph{$K_l$}  the number of these configurations that satisfy at least $l$ of the $i$-conditions, with $0\leq i\leq s$.
The binomial identity in Equation~\eqref{eq:binomial_identity} then represents a counting problem regarding this combinatorial model. 
More concretely, the LHS of Equation~\eqref{eq:binomial_identity} counts the total number of $j$-configurations, for $0\leq j\leq s$. That is, we count the configurations with repetitions (for the $j$-conditions) via $B_j$. 

\begin{example}
  In Figure~\ref{fig:all_config_byj}, we illustrate the corresponding list of configurations for our example according to the $j$-condition they satisfy, for $0\leq j \leq 2$. Moreover, we highlight in red one of the configurations to point out that the same configuration can be counted more than once. 
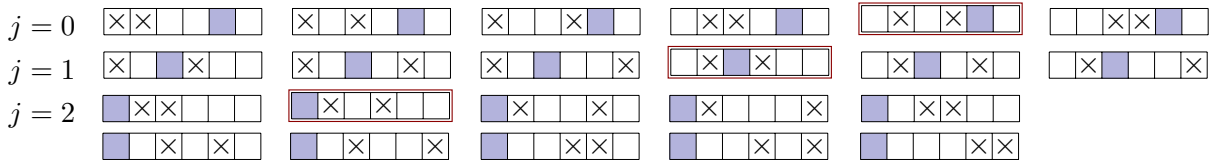
\begin{figure}[ht]
\centering
$    \begin{array}{rcccccc}
j=0 & 
\scalebox{0.35}{\begin{tikzpicture}
    \fill[DarkBlue!30] (4,0) rectangle (5,1);
   \draw[step=1cm,black,very thin] (0,0) grid (6,1);
   \node at (0.5,0.5) {\scalebox{3}{$\times$}};
   \node at (1.5,0.5) {\scalebox{3}{$\times$}};
\end{tikzpicture}} & 
\scalebox{0.35}{\begin{tikzpicture}
   \fill[DarkBlue!30] (4,0) rectangle (5,1);
   \draw[step=1cm,black,very thin] (0,0) grid (6,1);
   \node at (0.5,0.5) {\scalebox{3}{$\times$}};
   \node at (2.5,0.5) {\scalebox{3}{$\times$}};
\end{tikzpicture}} & 
\scalebox{0.35}{\begin{tikzpicture}
    \fill[DarkBlue!30] (4,0) rectangle (5,1);
   \draw[step=1cm,black,very thin] (0,0) grid (6,1);
   \node at (0.5,0.5) {\scalebox{3}{$\times$}};
   \node at (3.5,0.5) {\scalebox{3}{$\times$}};
\end{tikzpicture}} &
\scalebox{0.35}{\begin{tikzpicture}
    \fill[DarkBlue!30] (4,0) rectangle (5,1);
   \draw[step=1cm,black,very thin] (0,0) grid (6,1);
   \node at (1.5,0.5) {\scalebox{3}{$\times$}};
   \node at (2.5,0.5) {\scalebox{3}{$\times$}};
\end{tikzpicture}} & 
\scalebox{0.35}{\begin{tikzpicture}
    \draw[DarkRed, very thick] (-0.1,-0.1) rectangle (6.1,1.1); 
    \fill[DarkBlue!30] (4,0) rectangle (5,1);
   \draw[step=1cm,black,very thin] (0,0) grid (6,1);
   \node at (1.5,0.5) {\scalebox{3}{$\times$}};
   \node at (3.5,0.5) {\scalebox{3}{$\times$}};
\end{tikzpicture}} & 
\scalebox{0.35}{\begin{tikzpicture}
    \fill[DarkBlue!30] (4,0) rectangle (5,1);
   \draw[step=1cm,black,very thin] (0,0) grid (6,1);
   \node at (2.5,0.5) {\scalebox{3}{$\times$}};
   \node at (3.5,0.5) {\scalebox{3}{$\times$}};
\end{tikzpicture}}\\
j=1 & 
\scalebox{0.35}{\begin{tikzpicture}
    \fill[DarkBlue!30] (2,0) rectangle (3,1);
   \draw[step=1cm,black,very thin] (0,0) grid (6,1);
   \node at (0.5,0.5) {\scalebox{3}{$\times$}};
   \node at (3.5,0.5) {\scalebox{3}{$\times$}};
\end{tikzpicture}} & 
\scalebox{0.35}{\begin{tikzpicture}
   \fill[DarkBlue!30] (2,0) rectangle (3,1);
   \draw[step=1cm,black,very thin] (0,0) grid (6,1);
   \node at (0.5,0.5) {\scalebox{3}{$\times$}};
   \node at (4.5,0.5) {\scalebox{3}{$\times$}};
\end{tikzpicture}} & 
\scalebox{0.35}{\begin{tikzpicture}
    \fill[DarkBlue!30] (2,0) rectangle (3,1);
   \draw[step=1cm,black,very thin] (0,0) grid (6,1);
   \node at (0.5,0.5) {\scalebox{3}{$\times$}};
   \node at (5.5,0.5) {\scalebox{3}{$\times$}};
\end{tikzpicture}} &
\scalebox{0.35}{\begin{tikzpicture}
   \draw[DarkRed, very thick] (-0.1,-0.1) rectangle (6.1,1.1); 
    \fill[DarkBlue!30] (2,0) rectangle (3,1);
   \draw[step=1cm,black,very thin] (0,0) grid (6,1);
   \node at (1.5,0.5) {\scalebox{3}{$\times$}};
   \node at (3.5,0.5) {\scalebox{3}{$\times$}};
\end{tikzpicture}} & 
\scalebox{0.35}{\begin{tikzpicture}
    \fill[DarkBlue!30] (2,0) rectangle (3,1);
   \draw[step=1cm,black,very thin] (0,0) grid (6,1);
   \node at (1.5,0.5) {\scalebox{3}{$\times$}};
   \node at (4.5,0.5) {\scalebox{3}{$\times$}};
\end{tikzpicture}} & 
\scalebox{0.35}{\begin{tikzpicture}
    \fill[DarkBlue!30] (2,0) rectangle (3,1);
   \draw[step=1cm,black,very thin] (0,0) grid (6,1);
   \node at (1.5,0.5) {\scalebox{3}{$\times$}};
   \node at (5.5,0.5) {\scalebox{3}{$\times$}};
\end{tikzpicture}}\\
j=2 & 
\scalebox{0.35}{\begin{tikzpicture}
    \fill[DarkBlue!30] (0,0) rectangle (1,1);
   \draw[step=1cm,black,very thin] (0,0) grid (6,1);
   \node at (1.5,0.5) {\scalebox{3}{$\times$}};
   \node at (2.5,0.5) {\scalebox{3}{$\times$}};
\end{tikzpicture}} & 
\scalebox{0.35}{\begin{tikzpicture}
    \draw[DarkRed, very thick] (-0.1,-0.1) rectangle (6.1,1.1); 
   \fill[DarkBlue!30] (0,0) rectangle (1,1);
   \draw[step=1cm,black,very thin] (0,0) grid (6,1);
   \node at (1.5,0.5) {\scalebox{3}{$\times$}};
   \node at (3.5,0.5) {\scalebox{3}{$\times$}};
\end{tikzpicture}} & 
\scalebox{0.35}{\begin{tikzpicture}
    \fill[DarkBlue!30] (0,0) rectangle (1,1);
   \draw[step=1cm,black,very thin] (0,0) grid (6,1);
   \node at (1.5,0.5) {\scalebox{3}{$\times$}};
   \node at (4.5,0.5) {\scalebox{3}{$\times$}};
\end{tikzpicture}} &
\scalebox{0.35}{\begin{tikzpicture}
    \fill[DarkBlue!30] (0,0) rectangle (1,1);
   \draw[step=1cm,black,very thin] (0,0) grid (6,1);
   \node at (1.5,0.5) {\scalebox{3}{$\times$}};
   \node at (5.5,0.5) {\scalebox{3}{$\times$}};
\end{tikzpicture}} & 
\scalebox{0.35}{\begin{tikzpicture}
    \fill[DarkBlue!30] (0,0) rectangle (1,1);
   \draw[step=1cm,black,very thin] (0,0) grid (6,1);
   \node at (2.5,0.5) {\scalebox{3}{$\times$}};
   \node at (3.5,0.5) {\scalebox{3}{$\times$}};
\end{tikzpicture}} & 
\\
& \scalebox{0.35}{\begin{tikzpicture}
    \fill[DarkBlue!30] (0,0) rectangle (1,1);
   \draw[step=1cm,black,very thin] (0,0) grid (6,1);
   \node at (2.5,0.5) {\scalebox{3}{$\times$}};
   \node at (4.5,0.5) {\scalebox{3}{$\times$}};
\end{tikzpicture}} & 
\scalebox{0.35}{\begin{tikzpicture}
   \fill[DarkBlue!30] (0,0) rectangle (1,1);
   \draw[step=1cm,black,very thin] (0,0) grid (6,1);
   \node at (2.5,0.5) {\scalebox{3}{$\times$}};
   \node at (5.5,0.5) {\scalebox{3}{$\times$}};
\end{tikzpicture}} & 
\scalebox{0.35}{\begin{tikzpicture}
    \fill[DarkBlue!30] (0,0) rectangle (1,1);
   \draw[step=1cm,black,very thin] (0,0) grid (6,1);
   \node at (3.5,0.5) {\scalebox{3}{$\times$}};
   \node at (4.5,0.5) {\scalebox{3}{$\times$}};
\end{tikzpicture}} &
\scalebox{0.35}{\begin{tikzpicture}
    \fill[DarkBlue!30] (0,0) rectangle (1,1);
   \draw[step=1cm,black,very thin] (0,0) grid (6,1);
   \node at (3.5,0.5) {\scalebox{3}{$\times$}};
   \node at (5.5,0.5) {\scalebox{3}{$\times$}};
\end{tikzpicture}} & 
\scalebox{0.35}{\begin{tikzpicture}
    \fill[DarkBlue!30] (0,0) rectangle (1,1);
   \draw[step=1cm,black,very thin] (0,0) grid (6,1);
   \node at (4.5,0.5) {\scalebox{3}{$\times$}};
   \node at (5.5,0.5) {\scalebox{3}{$\times$}};
\end{tikzpicture}} & 
\\
    \end{array}$
    \caption{Configurations for $n=6$ and $s=k=2$ satisfying the $j$-condition.}
    \label{fig:all_config_byj}
\end{figure}
\end{example}

Note that once we fix $j$, if we want to construct a configuration that satisfies the $j$-condition, there are $s+k-2j$ boxes to the right of the barrier, from which we choose $j$ of them to be marked boxes, and there are $n-1-s-k+2j$ boxes to the left of the barrier, from which we choose $s-j$ of them to be marked boxes. Thus, we have the following result.
\begin{lemma}
For $0\leq j \leq s$, $\displaystyle{B_j =  \binom{s + k -2j}{s-j}\binom{n-1-s - k +2j}{j}}$.
\end{lemma}

For the RHS of Equation~\eqref{eq:binomial_identity}, we count the same set of configurations (with repetitions) but grouped by the minimum number of $i$-conditions they satisfy, with $0\leq i \leq s$. 
\begin{example}\label{ex:running example 3}
  In Figure~\ref{fig:all_config_byl}, we list all the configurations for our example according to the number $t$ of the $i$-conditions, with $0\leq i\leq 2$, that they satisfy, for $1\leq t \leq 3$. Then, we have that $K_3$ counts the configurations in row $t=3$, $\binom{6}{0}=1$; $K_2$ counts the configurations in rows $t=2$ and $3$, $\binom{6}{1}=6$; and $K_1$ the configurations in rows $t=1,2$, and $3$, $\binom{6}{2}=15$.
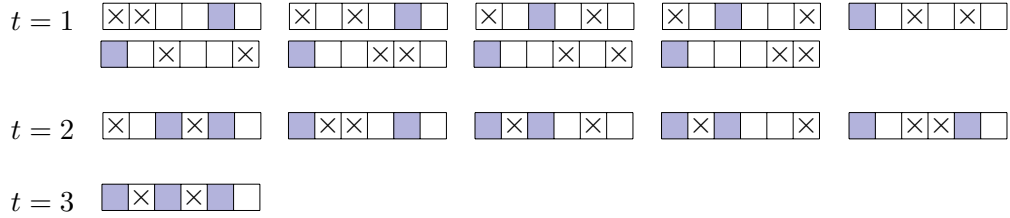
\begin{figure}[ht]
\centering
$\begin{array}{rccccc}
t=1 & 
\scalebox{0.35}{\begin{tikzpicture}
    \fill[DarkBlue!30] (4,0) rectangle (5,1);
   \draw[step=1cm,black,very thin] (0,0) grid (6,1);
   \node at (0.5,0.5) {\scalebox{3}{$\times$}};
   \node at (1.5,0.5) {\scalebox{3}{$\times$}};
\end{tikzpicture}} & 
\scalebox{0.35}{\begin{tikzpicture}
   \fill[DarkBlue!30] (4,0) rectangle (5,1);
   \draw[step=1cm,black,very thin] (0,0) grid (6,1);
   \node at (0.5,0.5) {\scalebox{3}{$\times$}};
   \node at (2.5,0.5) {\scalebox{3}{$\times$}};
\end{tikzpicture}} &
\scalebox{0.35}{\begin{tikzpicture}
   \fill[DarkBlue!30] (2,0) rectangle (3,1);
   \draw[step=1cm,black,very thin] (0,0) grid (6,1);
   \node at (0.5,0.5) {\scalebox{3}{$\times$}};
   \node at (4.5,0.5) {\scalebox{3}{$\times$}};
\end{tikzpicture}} &
\scalebox{0.35}{\begin{tikzpicture}
    \fill[DarkBlue!30] (2,0) rectangle (3,1);
   \draw[step=1cm,black,very thin] (0,0) grid (6,1);
   \node at (0.5,0.5) {\scalebox{3}{$\times$}};
   \node at (5.5,0.5) {\scalebox{3}{$\times$}};
\end{tikzpicture}} 
& \scalebox{0.35}{\begin{tikzpicture}
    \fill[DarkBlue!30]  (0,0) rectangle (1,1);
   \draw[step=1cm,black,very thin] (0,0) grid (6,1);
   \node at (2.5,0.5) {\scalebox{3}{$\times$}};
   \node at (4.5,0.5) {\scalebox{3}{$\times$}};
\end{tikzpicture}} \\ 
& 
\scalebox{0.35}{\begin{tikzpicture}
   \fill[DarkBlue!30]  (0,0) rectangle (1,1);
   \draw[step=1cm,black,very thin] (0,0) grid (6,1);
   \node at (2.5,0.5) {\scalebox{3}{$\times$}};
   \node at (5.5,0.5) {\scalebox{3}{$\times$}};
\end{tikzpicture}} & 
\scalebox{0.35}{\begin{tikzpicture}
    \fill[DarkBlue!30] (0,0) rectangle (1,1);
   \draw[step=1cm,black,very thin] (0,0) grid (6,1);
   \node at (3.5,0.5) {\scalebox{3}{$\times$}};
   \node at (4.5,0.5) {\scalebox{3}{$\times$}};
\end{tikzpicture}} &
\scalebox{0.35}{\begin{tikzpicture}
    \fill[DarkBlue!30] (0,0) rectangle (1,1);
   \draw[step=1cm,black,very thin] (0,0) grid (6,1);
   \node at (3.5,0.5) {\scalebox{3}{$\times$}};
   \node at (5.5,0.5) {\scalebox{3}{$\times$}};
\end{tikzpicture}} & 
\scalebox{0.35}{\begin{tikzpicture}
    \fill[DarkBlue!30] (0,0) rectangle (1,1);
   \draw[step=1cm,black,very thin] (0,0) grid (6,1);
   \node at (4.5,0.5) {\scalebox{3}{$\times$}};
   \node at (5.5,0.5) {\scalebox{3}{$\times$}};
\end{tikzpicture}} & 
\\[0.45cm] 
t=2 & 
\scalebox{0.35}{\begin{tikzpicture}
    \fill[DarkBlue!30] (2,0) rectangle (3,1);
    \fill[DarkBlue!30] (4,0) rectangle (5,1);
   \draw[step=1cm,black,very thin] (0,0) grid (6,1);
   \node at (0.5,0.5) {\scalebox{3}{$\times$}};
   \node at (3.5,0.5) {\scalebox{3}{$\times$}};
\end{tikzpicture}} & 
\scalebox{0.35}{\begin{tikzpicture}
    \fill[DarkBlue!30] (0,0) rectangle (1,1);
    \fill[DarkBlue!30] (4,0) rectangle (5,1);
   \draw[step=1cm,black,very thin] (0,0) grid (6,1);
   \node at (1.5,0.5) {\scalebox{3}{$\times$}};
   \node at (2.5,0.5) {\scalebox{3}{$\times$}};
\end{tikzpicture}} &
\scalebox{0.35}{\begin{tikzpicture}
    \fill[DarkBlue!30] (0,0) rectangle (1,1);
    \fill[DarkBlue!30] (2,0) rectangle (3,1);
   \draw[step=1cm,black,very thin] (0,0) grid (6,1);
   \node at (1.5,0.5) {\scalebox{3}{$\times$}};
   \node at (4.5,0.5) {\scalebox{3}{$\times$}};
\end{tikzpicture}} & 
\scalebox{0.35}{\begin{tikzpicture}
    \fill[DarkBlue!30] (0,0) rectangle (1,1);
    \fill[DarkBlue!30] (2,0) rectangle (3,1);
   \draw[step=1cm,black,very thin] (0,0) grid (6,1);
   \node at (1.5,0.5) {\scalebox{3}{$\times$}};
   \node at (5.5,0.5) {\scalebox{3}{$\times$}};
\end{tikzpicture}} &
\scalebox{0.35}{\begin{tikzpicture}
    \fill[DarkBlue!30] (0,0) rectangle (1,1);
    \fill[DarkBlue!30] (4,0) rectangle (5,1);
   \draw[step=1cm,black,very thin] (0,0) grid (6,1);
   \node at (2.5,0.5) {\scalebox{3}{$\times$}};
   \node at (3.5,0.5) {\scalebox{3}{$\times$}};
\end{tikzpicture}} \\[0.45cm] 
t=3 & 
\scalebox{0.35}{\begin{tikzpicture}
   \fill[DarkBlue!30] (0,0) rectangle (1,1);
   \fill[DarkBlue!30] (2,0) rectangle (3,1);
   \fill[DarkBlue!30] (4,0) rectangle (5,1);
   \draw[step=1cm,black,very thin] (0,0) grid (6,1);
   \node at (1.5,0.5) {\scalebox{3}{$\times$}};
   \node at (3.5,0.5) {\scalebox{3}{$\times$}};
\end{tikzpicture}} & & & & 
\end{array}$
    \caption{Configurations for $n=6$ and $s=k=2$ satisfying $t$  of the $i$-conditions.}
    \label{fig:all_config_byl}
\end{figure}
\end{example}

Since we count the same set of configurations, we have the following result. 
\begin{lemma}\label{lem:relation_coeffs_BK}
$\displaystyle{\sum_{j=0}^{s} B_j  = \sum_{l=0}^{s} K_{s+1-l}} = \sum_{l=1}^{s+1} K_{l} $.
\end{lemma}

Thus, the binomial identity in Equation~\eqref{eq:binomial_identity} is equivalent to the following result.
\begin{thm}\label{thm: final identity}
For $0 \leq l \leq s \leq \left\lfloor \frac{n-1}{2}\right\rfloor$,   $\displaystyle{K_{s+1-l} = \binom{n}{l} }.$
\end{thm}
We dedicate the rest of this section to proving this identity.

\subsection{From $i$-conditions to not at home marked boxes}

Our first step is to further develop the combinatorial model and find an alternative interpretation for $K_{s+1-l}$.
Recall that $K_{s+1-l}$ counts the number of configurations that satisfy at least $s+1-l$ of the $i$-conditions, with $0\leq i \leq s$, and that for the $i$-condition, the box at position $s+k+1-2i$ is a barrier. Now, given a configuration, we want to look at all the $i$-conditions at once. For that, we decorate all the configurations with $s$ marked boxes and $s+1$ barriers in positions $s+k+1-2i$, for $0\leq i \leq s$. We need to keep track of the position of the first barrier, and so let \demph{$b_0$}$:=k+1-s$. Note that since $s$ and $k$ are fixed parameters of our model and $b_0$ is determined by $s$ and $k$, then $b_0$ is fixed in our model too, and we will assume that it is part of the initial input data.

\begin{example}
In Figure~\ref{fig: config wrt b_0}, we illustrate the configurations of our running example, for which $b_0=1$.

\begin{figure}[ht]
\centering
$\begin{array}{ccccc}
\scalebox{0.35}{\begin{tikzpicture}
   \fill[DarkBlue!30] (0,0) rectangle (1,1);
   \fill[DarkBlue!30] (2,0) rectangle (3,1);
   \fill[DarkBlue!30] (4,0) rectangle (5,1);
   \draw[step=1cm,black,very thin] (0,0) grid (6,1);
   \node at (0.5,0.5) {\scalebox{3}{$\times$}};
   \node at (1.5,0.5) {\scalebox{3}{$\times$}};
\end{tikzpicture}} & 
\scalebox{0.35}{\begin{tikzpicture}
\fill[DarkBlue!30] (0,0) rectangle (1,1);
   \fill[DarkBlue!30] (2,0) rectangle (3,1);
   \fill[DarkBlue!30] (4,0) rectangle (5,1);
   \draw[step=1cm,black,very thin] (0,0) grid (6,1);
   \node at (0.5,0.5) {\scalebox{3}{$\times$}};
   \node at (2.5,0.5) {\scalebox{3}{$\times$}};
\end{tikzpicture}} &
\scalebox{0.35}{\begin{tikzpicture}
\fill[DarkBlue!30] (0,0) rectangle (1,1);
   \fill[DarkBlue!30] (2,0) rectangle (3,1);
   \fill[DarkBlue!30] (4,0) rectangle (5,1);
   \draw[step=1cm,black,very thin] (0,0) grid (6,1);
   \node at (0.5,0.5) {\scalebox{3}{$\times$}};
   \node at (3.5,0.5) {\scalebox{3}{$\times$}};
\end{tikzpicture}} &
\scalebox{0.35}{\begin{tikzpicture}
\fill[DarkBlue!30] (0,0) rectangle (1,1);
   \fill[DarkBlue!30] (2,0) rectangle (3,1);
   \fill[DarkBlue!30] (4,0) rectangle (5,1);
   \draw[step=1cm,black,very thin] (0,0) grid (6,1);
   \node at (0.5,0.5) {\scalebox{3}{$\times$}};
   \node at (4.5,0.5) {\scalebox{3}{$\times$}};
\end{tikzpicture}} &
\scalebox{0.35}{\begin{tikzpicture}
\fill[DarkBlue!30] (0,0) rectangle (1,1);
   \fill[DarkBlue!30] (2,0) rectangle (3,1);
   \fill[DarkBlue!30] (4,0) rectangle (5,1);
   \draw[step=1cm,black,very thin] (0,0) grid (6,1);
   \node at (0.5,0.5) {\scalebox{3}{$\times$}};
   \node at (5.5,0.5) {\scalebox{3}{$\times$}};
\end{tikzpicture}} \\
\scalebox{0.35}{\begin{tikzpicture}
\fill[DarkBlue!30] (0,0) rectangle (1,1);
   \fill[DarkBlue!30] (2,0) rectangle (3,1);
   \fill[DarkBlue!30] (4,0) rectangle (5,1);
   \draw[step=1cm,black,very thin] (0,0) grid (6,1);
   \node at (1.5,0.5) {\scalebox{3}{$\times$}};
   \node at (2.5,0.5) {\scalebox{3}{$\times$}};
\end{tikzpicture}} & 
\scalebox{0.35}{\begin{tikzpicture}
\fill[DarkBlue!30] (0,0) rectangle (1,1);
   \fill[DarkBlue!30] (2,0) rectangle (3,1);
   \fill[DarkBlue!30] (4,0) rectangle (5,1);
   \draw[step=1cm,black,very thin] (0,0) grid (6,1);
   \node at (1.5,0.5) {\scalebox{3}{$\times$}};
   \node at (3.5,0.5) {\scalebox{3}{$\times$}};
\end{tikzpicture}} &
\scalebox{0.35}{\begin{tikzpicture}
\fill[DarkBlue!30] (0,0) rectangle (1,1);
   \fill[DarkBlue!30] (2,0) rectangle (3,1);
   \fill[DarkBlue!30] (4,0) rectangle (5,1);
   \draw[step=1cm,black,very thin] (0,0) grid (6,1);
   \node at (1.5,0.5) {\scalebox{3}{$\times$}};
   \node at (4.5,0.5) {\scalebox{3}{$\times$}};
\end{tikzpicture}} & 
\scalebox{0.35}{\begin{tikzpicture}
\fill[DarkBlue!30] (0,0) rectangle (1,1);
   \fill[DarkBlue!30] (2,0) rectangle (3,1);
   \fill[DarkBlue!30] (4,0) rectangle (5,1);
   \draw[step=1cm,black,very thin] (0,0) grid (6,1);
   \node at (1.5,0.5) {\scalebox{3}{$\times$}};
   \node at (5.5,0.5) {\scalebox{3}{$\times$}};
\end{tikzpicture}} &
\scalebox{0.35}{\begin{tikzpicture}
\fill[DarkBlue!30] (0,0) rectangle (1,1);
   \fill[DarkBlue!30] (2,0) rectangle (3,1);
   \fill[DarkBlue!30] (4,0) rectangle (5,1);
   \draw[step=1cm,black,very thin] (0,0) grid (6,1);
   \node at (2.5,0.5) {\scalebox{3}{$\times$}};
   \node at (3.5,0.5) {\scalebox{3}{$\times$}};
\end{tikzpicture}} \\
\scalebox{0.35}{\begin{tikzpicture}
\fill[DarkBlue!30] (0,0) rectangle (1,1);
   \fill[DarkBlue!30] (2,0) rectangle (3,1);
   \fill[DarkBlue!30] (4,0) rectangle (5,1);
   \draw[step=1cm,black,very thin] (0,0) grid (6,1);
   \node at (2.5,0.5) {\scalebox{3}{$\times$}};
   \node at (4.5,0.5) {\scalebox{3}{$\times$}};
\end{tikzpicture}} &
\scalebox{0.35}{\begin{tikzpicture}
\fill[DarkBlue!30] (0,0) rectangle (1,1);
   \fill[DarkBlue!30] (2,0) rectangle (3,1);
   \fill[DarkBlue!30] (4,0) rectangle (5,1);
   \draw[step=1cm,black,very thin] (0,0) grid (6,1);
   \node at (2.5,0.5) {\scalebox{3}{$\times$}};
   \node at (5.5,0.5) {\scalebox{3}{$\times$}};
\end{tikzpicture}} &
\scalebox{0.35}{\begin{tikzpicture}
\fill[DarkBlue!30] (0,0) rectangle (1,1);
   \fill[DarkBlue!30] (2,0) rectangle (3,1);
   \fill[DarkBlue!30] (4,0) rectangle (5,1);
   \draw[step=1cm,black,very thin] (0,0) grid (6,1);
   \node at (3.5,0.5) {\scalebox{3}{$\times$}};
   \node at (4.5,0.5) {\scalebox{3}{$\times$}};
\end{tikzpicture}} &
\scalebox{0.35}{\begin{tikzpicture}
\fill[DarkBlue!30] (0,0) rectangle (1,1);
   \fill[DarkBlue!30] (2,0) rectangle (3,1);
   \fill[DarkBlue!30] (4,0) rectangle (5,1);
   \draw[step=1cm,black,very thin] (0,0) grid (6,1);
   \node at (3.5,0.5) {\scalebox{3}{$\times$}};
   \node at (5.5,0.5) {\scalebox{3}{$\times$}};
\end{tikzpicture}} & 
\scalebox{0.35}{\begin{tikzpicture}
\fill[DarkBlue!30] (0,0) rectangle (1,1);
   \fill[DarkBlue!30] (2,0) rectangle (3,1);
   \fill[DarkBlue!30] (4,0) rectangle (5,1);
   \draw[step=1cm,black,very thin] (0,0) grid (6,1);
   \node at (4.5,0.5) {\scalebox{3}{$\times$}};
   \node at (5.5,0.5) {\scalebox{3}{$\times$}};
\end{tikzpicture}}
\end{array}$
    \caption{Set of configurations for $n=6$ and $s=k=2$ together with the $3$ barriers starting at $b_0=1$.}
    \label{fig: config wrt b_0}
\end{figure}
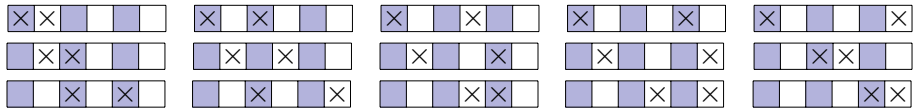
\end{example}

Since a marked box can be in the same box as a barrier, we introduce the following definition. Given a configuration $\gamma=(\gamma_1,\ldots, \gamma_s)$ and $b_0$, we say that the $j^{\text{th}}$ marked box, for $1\leq j \leq s$, is \demph{at home (with respect to $b_0$)} if $\gamma_j = b_0+2j-1$. That is if we label the barriers from left to right from $1^{\text{st}}$ to $(s+1)^{\text{th}}$, the $j^{\text{th}}$ marked box is at home if it is between the $j^{\text{th}}$ and $(j+1)^{\text{th}}$ barriers and we refer to this particular position as the \demph{$j^{\text{th}}$ home}. Otherwise, we say that the $j^{\text{th}}$ marked box is \demph{not at home (with respect to $b_0$)}. Given a configuration $\gamma$ and $b_0$, we denote by \demph{$nat(\gamma,b_0)$} the number of marked boxes \emph{not} at home in $\gamma$ with respect to $b_0$. We illustrate the relation between the homes, the barriers, and the conditions in Figure~\ref{fig:barriers&condition}.
We denote by \demph{$\mathcal{T}(n,s,i,b_0)$} the set of all configurations $\gamma$ with $n$ boxes in total and $s$ marked ones and such that $nat(\gamma)=i$ with respect to $b_0$, and by \demph{$T(n,s,i,b_0)$} the number of configurations in $\mathcal{T}(n,s,i,b_0)$. Since $b_0$ is fixed in our model, we omit the part of ``with respect to $b_0$'' when referring to at home or not at home marked boxes, but we keep $b_0$ in the notation as it plays an important role later on. 

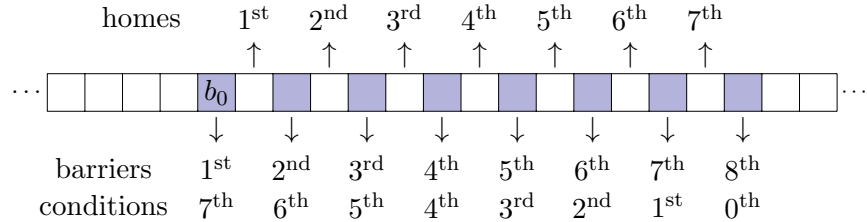
\begin{figure}[ht]
\centering
\scalebox{0.5}{
\begin{tikzpicture}
    \fill[DarkBlue!30] (4,0) rectangle (5,1);
    \fill[DarkBlue!30] (6,0) rectangle (7,1);
    \fill[DarkBlue!30] (8,0) rectangle (9,1);
    \fill[DarkBlue!30] (10,0) rectangle (11,1);
    \fill[DarkBlue!30] (12,0) rectangle (13,1);
    \fill[DarkBlue!30] (14,0) rectangle (15,1);
    \fill[DarkBlue!30] (16,0) rectangle (17,1);
    \fill[DarkBlue!30] (18,0) rectangle (19,1);
    \draw[step=1cm,black,very thin] (0,0) grid (21,1);
    \node at (-0.5,0.5){\scalebox{1.75}{$\cdots$}}; 
    \node at (21.5,0.5){\scalebox{1.5}{$\cdots$}}; 
    \node at (4.5,0.5){\scalebox{2}{$b_0$}};
    \node at (18.5,-0.5){\scalebox{2}{$\downarrow$}};
    \node at (16.5,-0.5){\scalebox{2}{$\downarrow$}};
    \node at (14.5,-0.5){\scalebox{2}{$\downarrow$}};
    \node at (12.5,-0.5){\scalebox{2}{$\downarrow$}};
    \node at (10.5,-0.5){\scalebox{2}{$\downarrow$}};
    \node at (8.5,-0.5){\scalebox{2}{$\downarrow$}};
    \node at (6.5,-0.5){\scalebox{2}{$\downarrow$}};
    \node at (4.5,-0.5){\scalebox{2}{$\downarrow$}};
    \node at (17.5,1.5){\scalebox{2}{$\uparrow$}};
    \node at (15.5,1.5){\scalebox{2}{$\uparrow$}};
    \node at (13.5,1.5){\scalebox{2}{$\uparrow$}};
    \node at (11.5,1.5){\scalebox{2}{$\uparrow$}};
    \node at (9.5,1.5){\scalebox{2}{$\uparrow$}};
    \node at (7.5,1.5){\scalebox{2}{$\uparrow$}};
    \node at (5.5,1.5){\scalebox{2}{$\uparrow$}};
    \node at (17.5,2.5){\scalebox{2}{$7^{\text{th}}$}};
    \node at (15.5,2.5){\scalebox{2}{$6^{\text{th}}$}};
    \node at (13.5,2.5){\scalebox{2}{$5^{\text{th}}$}};
    \node at (11.5,2.5){\scalebox{2}{$4^{\text{th}}$}};
    \node at (9.5,2.5){\scalebox{2}{$3^{\text{rd}}$}};
    \node at (7.5,2.5){\scalebox{2}{$2^{\text{nd}}$}};
    \node at (5.5,2.5){\scalebox{2}{$1^{\text{st}}$}};
    \node at (18.5,-1.5){\scalebox{2}{$8^{\text{th}}$}};
    \node at (16.5,-1.5){\scalebox{2}{$7^{\text{th}}$}};
    \node at (14.5,-1.5){\scalebox{2}{$6^{\text{th}}$}};
    \node at (12.5,-1.5){\scalebox{2}{$5^{\text{th}}$}};
    \node at (10.5,-1.5){\scalebox{2}{$4^{\text{th}}$}};
    \node at (8.5,-1.5){\scalebox{2}{$3^{\text{rd}}$}};
    \node at (6.5,-1.5){\scalebox{2}{$2^{\text{nd}}$}};
    \node at (4.5,-1.5){\scalebox{2}{$1^{\text{st}}$}};
    \node at (18.5,-2.5){\scalebox{2}{$0^{\text{th}}$}};
    \node at (16.5,-2.5){\scalebox{2}{$1^{\text{st}}$}};
    \node at (14.5,-2.5){\scalebox{2}{$2^{\text{nd}}$}};
    \node at (12.5,-2.5){\scalebox{2}{$3^{\text{rd}}$}};
    \node at (10.5,-2.5){\scalebox{2}{$4^{\text{th}}$}};
    \node at (8.5,-2.5){\scalebox{2}{$5^{\text{th}}$}};
    \node at (6.5,-2.5){\scalebox{2}{$6^{\text{th}}$}};
    \node at (4.5,-2.5){\scalebox{2}{$7^{\text{th}}$}};
    \node at (1.5,-1.5){\scalebox{2}{barriers}}; 
    \node at (2.5,2.5){\scalebox{2}{homes}}; 
    \node at (1.5,-2.5){\scalebox{2}{conditions}};
\end{tikzpicture}
}
    \caption{The relation of the barriers and conditions.}
    \label{fig:barriers&condition}
\end{figure}

The next two results provide some technical information regarding the marked boxes, barriers, and conditions given a configuration. The first result gives us some cases in which a configuration satisfies a particular condition.
\begin{lemma}\label{lem: emptybarrier}
    Consider a configuration $\gamma$ and suppose that, for some $i\leq s$, the $i^{\text{th}}$ marked box is at home. Then, we have that:
   \begin{itemize}
        \item if the $i^{\text{th}}$ barrier is empty, then $\gamma$ satisfies the $(s-i+1)$-condition; and
        \item if the $(i+1)^{\text{th}}$ barrier is empty, then $\gamma$ satisfies the $(s-i)$-condition.
    \end{itemize}
\end{lemma}
Note that if both barriers are empty, then $\gamma$ satisfies both the $(s-i)$ and $(s-i-1)$-conditions. The idea of the lemma is sketched in Figure~\ref{fig:b0lemma}.   

\begin{figure}[ht]
\centering
\scalebox{0.5}{
\begin{tikzpicture}
    \fill[DarkBlue!30] (4,0) rectangle (5,1);
    \fill[DarkBlue!30] (6,0) rectangle (7,1);
    \fill[DarkBlue!30] (8,0) rectangle (9,1);
    \fill[DarkBlue!30] (10,0) rectangle (11,1);
    \fill[DarkBlue!30] (12,0) rectangle (13,1);
    \fill[DarkBlue!30] (14,0) rectangle (15,1);
    \fill[DarkBlue!30] (16,0) rectangle (17,1);
    \fill[DarkBlue!30] (18,0) rectangle (19,1);
    \draw[step=1cm,black,very thin] (0,0) grid (23,1);
    \node at (2.5,0.5) {\scalebox{1.75}{$\cdots$}}; 
    \node at (5.5,0.5){\scalebox{1.75}{$\cdots$}};
    \node at (15.5,0.5){\scalebox{1.75}{$\cdots$}};
    \node at (20.5,0.5){\scalebox{1.5}{$\cdots$}}; 
    \node at (4.5,0.5){\scalebox{2}{$b_0$}};
    \node at (8.5,0.5){\scalebox{3}{$\times$}};
    \node at (11.5,0.5){\scalebox{3}{$\times$}};
    \node at (5.5,-2.5){\scalebox{2}{marked box at home}};
    \node at (6.5,-1.5){\scalebox{2}{barriers}}; 
    \node at (10.5,-0.5){\scalebox{2}{$\downarrow$}};
    \node at (12.5,-0.5){\scalebox{2}{$\downarrow$}};
    \node at (11.5, -1.1){\scalebox{2}{$\bigg\downarrow$}};
    \node at (10.5,-1.5){\scalebox{1.75}{$i^{\text{th}}$}};
    \node at (12.5,-1.5){\scalebox{1.75}{$\qquad (i+1)^{\text{th}}$}};
    \node at (11.5,-2.5){\scalebox{1.75}{$i^{\text{th}}$}};
    \node at (-1,2){\scalebox{2}{$\#$ of marked boxes}};
    \draw[|-|] (0,1.5) -- (11,1.5);
    \node at (5.5,2){\scalebox{1.75}{$i-1$}};
    \draw[|-|] (12,1.5) -- (23,1.5);
    \node at (17.5,2){\scalebox{1.75}{$s-i$}};
\end{tikzpicture}}
    \caption{Sketch of a configuration $\gamma$ illustrating the statements in Lemma~\ref{lem: emptybarrier}.}
    \label{fig:b0lemma}
\end{figure}
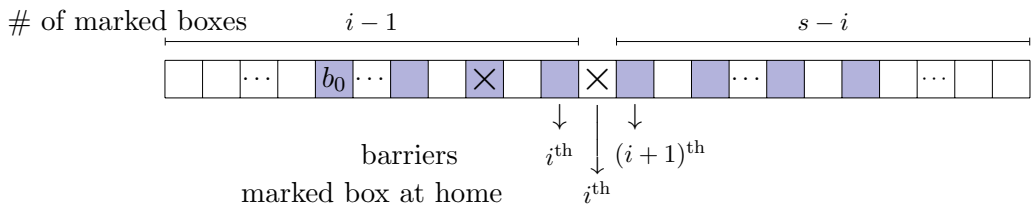

The second result gives us sufficient conditions for a configuration not to satisfy a particular condition.
\begin{lemma}\label{lem: notathome}
   Consider a configuration $\gamma$. Then, we have that:    
\begin{itemize}
        \item if the $i^{\text{th}}$ barrier has a marked box, then $\gamma$ does not satisfy the $(s-i+1)$-condition; and

        \item if the $j^{\text{th}}$ marked box is in the $i^{\text{th}}$ home and $j \neq i$, then $\gamma$ satisfies neither condition associated with the barriers adjacent to that marked box.
    \end{itemize}
    
\end{lemma}
\begin{proof}
The first statement follows from the fact that the $(s-i+1)$-condition is the condition associated with the $i^{\text{th}}$ barrier, which contains a mark. This idea is sketched in the top picture in Figure~\ref{fig:sketchLemmab0IB}.

For the second statement, if either barrier is not empty, then their associated conditions are not satisfied. Assuming they are both empty, there are $j-1 \neq i -1$ marks to the left of the $i^{\text{th}}$ barrier, and $s-j \neq s - i$ marks to the right of the $(i+1)^{\text{th}}$ barrier. Since there needs to be $i - 1$ marks to the left of the $i^{\text{th}}$ barrier and $s-(i+1) + 1 = s - i$ marks to the right of the $(i+1)^{\text{th}}$ barrier for their respective conditions to be satisfied, neither condition is satisfied. We sketch the idea in the bottom picture in Figure~\ref{fig:sketchLemmab0IB}.

\begin{figure}[ht]
\centering
\scalebox{0.5}{
\begin{tikzpicture}
    \fill[DarkBlue!30] (4,0) rectangle (5,1);
    \fill[DarkBlue!30] (6,0) rectangle (7,1);
    \fill[DarkBlue!30] (8,0) rectangle (9,1);
    \fill[DarkBlue!30] (10,0) rectangle (11,1);
    \fill[DarkBlue!30] (12,0) rectangle (13,1);
    \fill[DarkBlue!30] (14,0) rectangle (15,1);
    \fill[DarkBlue!30] (16,0) rectangle (17,1);
    \fill[DarkBlue!30] (18,0) rectangle (19,1);
    \draw[step=1cm,black,very thin] (0,0) grid (23,1);
    \node at (2.5,0.5){\scalebox{1.75}{$\cdots$}}; 
    \node at (5.5,0.5){\scalebox{1.75}{$\cdots$}};
    \node at (15.5,0.5){\scalebox{1.75}{$\cdots$}};
    \node at (20.5,0.5){\scalebox{1.5}{$\cdots$}}; 
    \node at (4.5,0.5){\scalebox{2}{$b_0$}};
    \node at (8.5,0.5){\scalebox{3}{$\times$}};
    \node at (11.5,0.5){\scalebox{3}{$\times$}};
    \node at (4.5,-1.5){\scalebox{2}{barriers}}; 
    \node at (4.5,-2.5){\scalebox{2}{marked box}};
    \node at (8.5,-0.5){\scalebox{2}{$\downarrow$}};
    \node at (8.5,-1.5){\scalebox{2}{$i^{\text{th}}$}};
    \node at (8.5,-2.5){\scalebox{2}{$j^{\text{th}}$}};
    \node at (-1,2){\scalebox{2}{$\#$ of marked boxes}};
    \draw[|-|] (0,1.5) -- (8,1.5);
    \node at (5.5,2){\scalebox{1.75}{$j-1$}};
    \draw[|-|] (9,1.5) -- (23,1.5);
    \node at (17.5,2){\scalebox{1.75}{$s-j$}};
\end{tikzpicture}
}

\bigskip

\scalebox{0.5}{
\begin{tikzpicture}
    \fill[DarkBlue!30] (4,0) rectangle (5,1);
    \fill[DarkBlue!30] (6,0) rectangle (7,1);
    \fill[DarkBlue!30] (8,0) rectangle (9,1);
    \fill[DarkBlue!30] (10,0) rectangle (11,1);
    \fill[DarkBlue!30] (12,0) rectangle (13,1);
    \fill[DarkBlue!30] (14,0) rectangle (15,1);
    \fill[DarkBlue!30] (16,0) rectangle (17,1);
    \fill[DarkBlue!30] (18,0) rectangle (19,1);
    \draw[step=1cm,black,very thin] (0,0) grid (23,1);
    \node at (2.5,0.5) {\scalebox{1.75}{$\cdots$}}; 
    \node at (5.5,0.5){\scalebox{1.75}{$\cdots$}};
    \node at (15.5,0.5){\scalebox{1.75}{$\cdots$}};
    \node at (20.5,0.5){\scalebox{1.5}{$\cdots$}}; 
    \node at (4.5,0.5){\scalebox{2}{$b_0$}};
    \node at (8.5,0.5){\scalebox{3}{$\times$}};
    \node at (11.5,0.5){\scalebox{3}{$\times$}};
    \node at (6.5,-2.5){\scalebox{2}{marked box}};
    \node at (6.5,-1.5){\scalebox{2}{barriers}}; 
    \node at (10.5,-0.5){\scalebox{2}{$\downarrow$}};
    \node at (12.5,-0.5){\scalebox{2}{$\downarrow$}};
    \node at (10.5,-1.5){\scalebox{2}{$i^{\text{th}}$}};
    \node at (12.5,-1.5){\scalebox{1.75}{$\quad (i+1)^{\text{th}}$}};
    \node at (11.5,-2.5){\scalebox{2}{$j^{\text{th}}$}};
    \node at (-1,2){\scalebox{2}{$\#$ of marked boxes}};
    \draw[|-|] (0,1.5) -- (11,1.5);
    \node at (5.5,2){\scalebox{1.75}{$j-1$}};
    \draw[|-|] (12,1.5) -- (23,1.5);
    \node at (17.5,2){\scalebox{1.75}{$s-j$}};
\end{tikzpicture}
}
    \caption{Sketch of two generic configurations illustrating the statements in Lemma~\ref{lem: notathome}.}
    \label{fig:sketchLemmab0IB}
\end{figure}
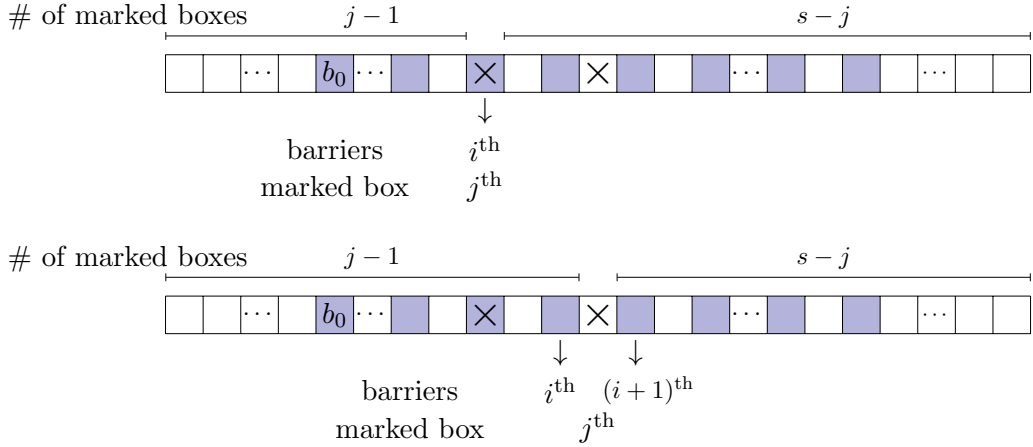
\end{proof}

Lemmas~\ref{lem: emptybarrier} and~\ref{lem: notathome} provide the relation between the $j$-conditions and the not at home marked boxes, allowing us to state the following result.

\begin{prop}\label{thm:conditions and marked boxes}
    For $s\leq k \leq \left\lfloor \frac{n-1}{2}\right\rfloor$ and $b_0 = k-s+1$, $T(n,s,i,b_0)$ also counts the number of configurations $\gamma$ with $n$ boxes in total and $s$ marked ones that satisfy exactly $s+1-i$ conditions.
\end{prop}

\begin{proof}
Let $\gamma$ be the unique configuration in $\mathcal{T}(n,s,0,b_0)$. Then, all the marked boxes are at home, and $\gamma$ satisfies all the $j$-conditions, with $0\leq j\leq s$, and the statement is true. 

Now, suppose that $i>0$. A configuration $\beta\in \mathcal{T}(n,s,i,b_0)$ can be obtained from a configuration $\lambda \in \mathcal{T}(n,s,0,b_0)$ by moving the at home marked boxes to positions so that they are no longer at home. We assume the marked box moves one box at a time. Then, we have the following relation between the number of not at home marked boxes and the number of $j$-conditions it satisfies in a configuration after moving one marked box:
\begin{enumerate}
    \item Moving an at home marked box increases the number of not at home marked boxes by 1 and decreases the number of $j$-conditions it satisfies by 1. 

    Suppose we move the $i^{\text{th}}$ marked box. Since it is at home, it is between the $i^{\text{th}}$ and the $(i+1)^{\text{th}}$ barriers. Moving it to either of the barriers increases the number of 'not at home' marked boxes by 1, since the other marked boxes remain in the same position. Moreover, the new configuration does not satisfy the $(s+1-i)$-condition as it has one box in the barrier. 

    \item Moving a not at home marked box into a home decreases the number of not at home marked boxes by 1 and increases the number of $j$-conditions it satisfies by 1.

    Suppose we move the $i^{\text{th}}$ marked box. It was either in the $i^{\text{th}}$ barrier or the $(i+1)^{\text{th}}$ barrier to start. Either way, the starting configuration does not satisfy the condition associated with the barrier the mark was originally in. Now, since there are $i-1$ marks before this mark and $n-i$ marks after it, moving this mark into its home means that the condition associated with the barrier the mark was originally in is now satisfied. This means that the not at home count decreased by 1 with this move, and the number of conditions satisfied increased by 1.

    \item Moving a not at home marked box in a way that continues to be not at home keeps the number of $j$-conditions and the number of not at home marked boxes the same.  

    This case corresponds to moving a marked box from a barrier to an empty box that is not its home, or vice versa. We leave it to the reader to verify this claim.

\end{enumerate}
\end{proof}

Therefore, we have the following interpretation of $K_{s+1-l}$.
\begin{cor}
For $s\leq k \leq \left\lfloor \frac{n-1}{2}\right\rfloor$, $0 \leq l \leq s$, and $b_0 = k-s+1$,
$$
K_{s+1-l} = \displaystyle{\sum_{i=0}^l T(n,s,i,b_0)}.
$$
\end{cor}

Our next goal is to prove the following closed formula for $T(n,s,i,b_0)$. 
\begin{thm}\label{thm: closed form of $T(n,i)$}
Let $s\leq \left\lfloor \frac n2 \right\rfloor$. For $i=0$, $T(n,s,0,b_0) = 1$, and for $0< i \leq s$, 
$$
T(n,s,i,b_0) = \binom{n}{i}-\binom{n}{i-1}.
$$
\end{thm}
\begin{remark}
For $n$ odd, we note that $s= \left\lfloor \frac n2 \right\rfloor$ implies that $s =\left\lfloor \frac{n-1}{2} \right\rfloor$. For $n$ even, $s =\left\lfloor \frac n2 \right\rfloor$ implies that $b_0=1$ and so exactly half of the boxes are marked, and the barriers are located in alternating positions starting at the $1^{\text{st}}$ position. Thus, we consider the $(s+1)^{\text{th}}$ home to be the last box in the configuration, so that the combinatorial model makes sense and we can count the number of configurations in this particular scenario, which will be relevant later. 
\end{remark}
Before proving Theorem~\ref{thm: closed form of $T(n,i)$}, we finish the proof of Equation~\eqref{eq:binomial_identity} in Theorem~\ref{thm:binomial_identity}, which we have reduced to the following result.
\begin{thm*}
For $0 \leq l \leq s \leq \left\lfloor \frac{n-1}{2}\right\rfloor$,   $\displaystyle{K_{s+1-l} = \binom{n}{l} }.$
\end{thm*} 

\begin{proof}[Proof of Theorem~\ref{thm: final identity}]
We have that
    \begin{align*}
K_{s+1-l} &= \sum_{i=0}^l T(n,s,i,b_0) 
= 1 + \sum_{i=1}^l T(n,s,i,b_0) = 1 + \sum_{i=1}^l \left[\binom{n}{i}-\binom{n}{i-1}\right] = \binom{n}{l}, 
    \end{align*}
where we use that we have a telescopic summation and all the terms except one cancel. 
\end{proof}

\subsection{Independency of $b_0$}\label{subsect: independency of b_0}

In this subsection, we first show that, for computing $T(n,s,i,b_0)$, we can assume that the barriers start at $b_0=1$. We start with some definitions. 

Given a configuration $\gamma= (\gamma_1,\ldots,\gamma_s)$ and $1\leq j\leq s$, let $d_j$ be the number of boxes from the position of the $j^{\text{th}}$ marked box to the end, $d_j := n-\gamma_j$, and let $a_j := 2(s-j)$. Moreover, let $j_0$ be the rightmost marked box such that $d_j >a_j$, with $j_0=0$ if no such box exists. Then, we define the \demph{flipped block} of $\gamma$, $FB(\gamma)$, as the set of boxes to the right of the $(j_0+1)^{\text{th}}$ marked box (including that marked box), that is $FB(\gamma) = (\gamma_{j_0+1}, \ldots, \gamma_s)$, of length $n-\gamma_{j_0+1}+1$. Denote by $m:=s-j_0$ the number of marked boxes on it. Finally, we denote by $\gamma^0$ the marked boxes not in the flipped block, that is $\gamma = \gamma^0\cup FB(\gamma)$, where the union of sentences is defined as the concatenation of the two sentences. Note that if $j_0=0$, then $FB(\gamma) = \gamma$, and that if $j_0=s$, then $FB(\gamma)=\emptyset$. 

Similarly, we can define the \demph{reverse flipped block} of $\gamma$, $FB'(\gamma)$ in the following way. Define $d_j'$ to be the index of the $j^{\text{th}}$ marked box, i.e. $d_j' := \gamma_j-1$, and let $a_j':= 2(j-1)$. Now, we define $j_0'$ to be the leftmost marked box such that $d_j' > a_j'$. If no such box exists, then $j_0' = n+1$. We then define $FB'(\gamma)$ as the set of boxes to the left of the $(j_0'-1)^{\text{th}}$ marked box, including that marked box. Denote by $m' := j_0' - 1$ the number of marked boxes in the reverse flipped block, and denote by $\gamma'^0$ the boxes in $\gamma'$ not in the reverse flipped block. Thus, $\gamma = FB'(\gamma) \cup \gamma'^0$. Note that if $j_0' = 1$, then $FB'(\gamma) = \emptyset$, and that if $j_0' = n+1$, then $FB'(\gamma) = \gamma$.

The following result provides a bijection between configurations with the barriers starting at $b_0$ and configurations with the barriers starting at $b_0+1$.
\begin{prop}\label{prop:independency of b_0}
 For $0\leq i \leq s \leq \left\lfloor \frac{n-1}{2} \right\rfloor$, we have that $T(n,s,i,b_0) = T(n,s,i,b_0+1)$. In particular, $T(n,s,i,b_0) = T(n,s,i,1)$. 
\end{prop}

\begin{proof}
Since the parameters $n$, $s$, and $i$ are fixed, we simplify the notation in this proof and denote by $\mathcal{T}(b_0)$ the set of all configurations $\gamma$ with $s$ marked boxes and length $n$ and such that $nat(\gamma)=i$ with respect to $b_0$. We show that $T(n,s,i,b_0) = T(n,s,i,b_0+1)$ by giving a bijection:
\begin{equation*}
    \begin{array}{rccc}
    \phi: & \mathcal{T}(b_0) & \longrightarrow & \mathcal{T}(b_0+1) \\
     & \gamma & \longmapsto & \gamma':= \phi(\gamma)
    \end{array}
\end{equation*}
Intuitively, $\phi$ reflects the pattern of marked boxes in the flipped block at the beginning of the configuration and moves the rest of the marked boxes to the right by $2m+1$ spots. For $\gamma = (\gamma_1,\ldots, \gamma_s)$, we define
\begin{equation*}
    \phi(\gamma_j):= \left\{ 
    \begin{array}{lcl}
    1+d_j, & \qquad & \text{if } \gamma_j\in FB(\gamma), \\
    \gamma_j +2m+1, & \qquad & \text{if } \gamma_j\notin FB(\gamma),
    \end{array}\right.
\end{equation*}
so that $\gamma'$ is the sequence obtained by reordering in increasing order the entries $\phi(\gamma_j)$.

Let us start by showing that $\phi$ is well-defined. That is, we show that $\gamma'$ is a configuration and that it has the same number of market boxes not at home as $\gamma$. If $FB(\gamma)=\gamma$, then $d(\gamma)=(d_1,\ldots, d_s)$ is a strictly decreasing sequence of integers, and so $\gamma' = (d_s+1,\ldots,d_1+1)$ is also a strictly increasing sequence of integers. 
If $FB(\gamma) \neq \gamma$, then $\gamma = \gamma^0 \cup FB(\gamma)$, with $\gamma^0 = \left(\gamma_1, \ldots, \gamma_{j_0}\right)$ and $FB(\gamma)$ nonempty sets. Then, $\phi(\gamma)= \phi(FB(\gamma)) \cup \phi(\gamma^0)$, where $\phi(FB(\gamma))$ corresponds to the marked boxes in the first $2m+1$ boxes of $\gamma'$ and $\phi(\gamma^0)$ corresponds to the marked boxes in the last $n-(2m+1)$ boxes of $\gamma'$. Thus, the values of $\phi(\gamma^0)$ and $\phi(FB(\gamma))$ do not overlap and $\gamma'$ is a configuration with $s$ marked boxes and length $n$. 

Next, we show that the number of marked boxes not at home is preserved under $\phi$. That is, if $nat(\gamma)=i$ with respect to $b_0$, then $nat(\gamma')=i$ with respect to $b_0+1$. Note that in $\gamma$ the $j^{\text{th}}$ home corresponds to the box in position $b_0+2j-1$ while in $\gamma'$ the $j^{\text{th}}$ home corresponds to the box in position  $b_0+2j$.

Consider the $i^{\text{th}}$ marked box in $\gamma$. Then, we have two cases:
\begin{itemize}
\itemsep0.2in
    \item \underline{Case 1:} $\gamma_i \in FB(\gamma)$. \\
    In this case, $\gamma_i$ is a marked box not at home in $\gamma$ and $\phi(\gamma_i)$ is also a marked box not at home in $\gamma'$. 
    By construction, the marked boxes in the flipped block satisfy that $d_{i} \leq a_{i}=2(s-i)$, and the $i^{\text{th}}$ home corresponds to the position $b_0+2i-1$, for which $d_{b_0+2i-1} = n-b_0-2i+1$. However, we have that $b_0\leq n-2s$ since there needs to be enough space to the right of $b_0$ to have $s$ barriers. Therefore, we have
    $$d_{b_0+2i-1} = n-b_0-2i+1 \geq n - (n-2s)-2i+1 = 2(s-i)+1 > 2(s-i)\geq d_i;
    $$
    This means that the $i^{\text{th}}$ home is strictly to the left of the $i^{\text{th}}$ marked box, which implies that the $i^{\text{th}}$ marked box cannot be at home. Similarly, for $\phi(\gamma_i)= \gamma'_{s-i+1}$ we have that 
    $$b_0 + 1 +2(s-i+1) = b_0+2(s-i)+3 \geq b_0+3 + d_i >d_i+1 = \phi(\gamma_i) = \gamma'_{s-i+1}.
    $$
    Thus, the position of the $(s-i+1)^{\text{th}}$ home is strictly to the right of the $\gamma'_{s-i+1}$ marked box, which corresponds to the image of $\gamma_i$. Therefore, $\phi(\gamma_i)$ is also not at home in $\gamma'$. 

    \item \underline{Case 2:} $\gamma_i \notin FB(\gamma)$. \\
    In this case, $\phi(\gamma_j)= \gamma'_{j+m}$ and $\gamma_j$ is at home if and only if $\gamma'_{j+m}$ is at home. 

    Note that $\gamma'_{j+m}= \gamma_j + 2m+1$. Moreover, $\gamma_j$ is at home if and only if $\gamma_j = b_0+2j-1$. Similarly, $\gamma'_{j+m}$ is at home if and only if $\gamma'_{j+m} = b_0+2(j+m)$. Putting these three observations together, we have that 
\begin{align*}
\gamma_j \text{ is at home } &\Longleftrightarrow \gamma_j = b_0+2j-1 \Longleftrightarrow \gamma_j+2m+1 = b_0+2j-1+2m+1 \\
&\Longleftrightarrow \gamma'_{j+m} = b_0+2(j+m) \Longleftrightarrow \gamma'_{j+m} \text{ is at home.}    
\end{align*}
\end{itemize}

The details about the inverse map are left to the reader.
\end{proof}

\begin{remark}
    For the rest of this section, we assume that $b_0=1$, and so we abbreviate $\mathcal{T}(n,s,i,b_0)$ to $\mathcal{T}(n,s,i)$ and $T(n,s,i,b_0)$ to $T(n,s,i)$.
\end{remark}

\subsection{Recursive relations}
In this subsection, we study some recursive relations. We start with the general case.   
\begin{prop}\label{thm: recursion of T}
    For $0<i \leq s \leq \lfloor\frac{n-1}{2}\rfloor$,
    \begin{equation}\label{eq: recursive formula}
    T(n,s,i) = T(n-1,s,i) + T(n-1,s-1,i-1).
    \end{equation}
\end{prop}

\begin{proof}
The recursive formula in Equation~\eqref{eq: recursive formula} follows from the following counting argument. Given $\gamma \in \mathcal{T}(n,s,i)$, let $\gamma'$ be the configuration obtained by deleting the last box in $\gamma$. Then, either:
\begin{itemize}
    \item the last box of $\gamma$ is a marked box, and so $\gamma'\in \mathcal{T}(n-1,s-1,i-1)$; or
    \item the last box of $\gamma$ is not a marked box, and so $\gamma'\in \mathcal{T}(n-1, s, i)$. 
\end{itemize}
\end{proof}

It is important to note that Proposition~\ref{thm: recursion of T} does not apply to the case $T(2s,s,s)$ as $s > \lfloor \frac{n-1}{2}\rfloor$. Thus, we need to explore this case separately.
We begin with a technical lemma that states that for configurations in $\mathcal{T}(2i, i, i)$, the first marked box has to be in the first position. 
\begin{lemma}
For $\gamma=(\gamma_1,\ldots, \gamma_i)  \in \mathcal{T}(2i, i, i)$, with $i\geq 1$, $\gamma_1=1$.
\end{lemma}
\begin{proof}
We proceed by induction on $i$. For $i=1$, the statement is true since $\mathcal{T}(2, 1, 1)$ only contains the configuration $\gamma=(1)$. 

Next, assume the statement is true for all configurations in $\mathcal{T}(2i-2, i-1, i-1)$. By contradiction, suppose there exists $\gamma =(\gamma_1,\ldots, \gamma_i) \in \mathcal{T}(2i, i, i)$ such that $\gamma_1>1$. In fact, $\gamma_1>2$ since otherwise the first marked box will be in the first home, and there cannot be any at home marked box. Now, let $\gamma'\in \mathcal{T}(2i-2, i-1, i-1)$ be the configuration $\gamma'=(\gamma_2-2,\ldots, \gamma_i-2)$, obtained from $\gamma$ by ignoring the first two boxes of $\gamma$ and deleting the first marked box. For any $k\geq 2$, we know that the $k^{\text{th}}$ marked box of $\gamma$ is not in the $(2k)^{\text{th}}$ box. This implies that the $(k-1)^{\text{th}}$ marked box of $\gamma'$ is not in the $(2k-2)^{\text{th}}$ box of $\gamma'$, and so every marked box of $\gamma'$ is not at home. Since $\gamma'$ is a configuration in $\mathcal{T}(2i-2, i-1, i-1)$, by induction hypothesis, $\gamma'$ has a marked box in its first box. However, by construction, the first marked box of $\gamma'$ must be in its second box or later. Thus, we get a contradiction, and the result follows. 
\end{proof}

Next, we prove the recursive relation from Proposition~\ref{thm: recursion of T} for $T(2i,i,i)$.
\begin{thm}\label{thm: recursion of T- particular case}
For $i\geq 1$,    $T(2i,i,i) = T(2i-1,i-1,i-1)$.
\end{thm}
\begin{proof}
We show that there is a bijection $\phi$ between $\mathcal{T}(2i,i,i)$ and $\mathcal{T}(2i-1,i-1,i-1)$. Consider $\gamma = (\gamma_1, \ldots, \gamma_i)\in \mathcal{T}(2i,i,i)$. We split $\gamma$ into two configurations, $\gamma_c$ and $\gamma_r$, so that $\gamma_c=(\gamma_1,\ldots, \gamma_j)$ is the smallest subsequence such that the configuration has length $2j$ and $j$ marked boxes. We denote by $\gamma_r$ the rest of $\gamma$, $\gamma_r = (\gamma_{j+1}, \ldots, \gamma_i)$, with $\gamma_{k+1}> 2j$, and refer to it as the \emph{reflected block}. 

Now, we are ready to define the bijection. For $\gamma = (\gamma_1, \ldots, \gamma_i)\in \mathcal{T}(2i,i,i)$, the configuration $\phi(\gamma)=\beta = (\beta_1, \ldots, \beta_{i-1})$ is defined by the following algorithm:
\begin{itemize}
    \item The marks in $\gamma_c$ do not move.
    \item The marks in $\gamma_r$ are reflected so that the $(j+k)^{\text{th}}$ mark moves from the $(2j+m)^{\text{th}}$ box to the $(2i-m+1)^{\text{th}}$ box. 
    \item Delete the first marked box of $\gamma$, which is in the first box.
    \item The resulting configuration is $\beta$. 
\end{itemize}
We need to show that $\beta$ is in fact a configuration in $\mathcal{T}(2i-1,i-1,i-1)$. By construction, it is clear that $\beta$ is a configuration of length $2i-1$ with $i-1$ marked boxes. Thus, we only need to see the number of not at home marked boxes in $\beta$. 

For $\gamma_c$, we notice that the marked boxes in this part of the configuration are not at home as marked boxes in $\gamma$ since the $k^{\text{th}}$ marked box is in a position before the $(2k-1)^{\text{th}}$ box, and so in $\beta$, the $(k-1)^{\text{th}}$ marked box is in a position before the $(2k-2)^{\text{th}}$ box. For $\gamma_r$, we want to show that the number of not at home is preserved. Let $(j+k)^{\text{th}}$ be a marked box in $\gamma_r$ in the $(2j+m)^{\text{th}}$ box. If it is not at home, then it is not in the $(2j+2k)^{\text{th}}$, and so $m\neq 2k$. When we apply $\phi$, that box is now the $(i-k)^{\text{th}}$ marked box in $\beta$ and it is in the $(2i-m)^{\text{th}}$ box, which is not the $(2i-2k)^{\text{th}}$ box since $m\neq 2k$. Thus, $\phi$ does preserve the $nat$ statistic, and it is well-defined.

The details about the inverse map are left to the reader.
\end{proof}

\subsection{Independency of $s$}\label{subsect: independency of s}
In this subsection, we show that, for computing $T(n,s,i)$, we can assume that $s=\left\lfloor \frac{n}{2} \right\rfloor$. We start looking at the special case. Next result tells us how to compute $T(2l,l,l-1)$ in terms of the configurations of smaller size, depending on the position of the last marked box. 

\begin{thm} \label{thm: exactly k + l + 1}
Given $l$ and $i<l$, the number of configurations in $\mathcal{T}(2l, l, l-1)$ with last marked box in the $(l+i+1)^{\text{th}}$ box is equal to $T(l+i,i,i)$.
\end{thm}

\begin{proof}
We start by noticing that for $l>1$, any configuration $\gamma=(\gamma_1,\ldots, \gamma_l) \in \mathcal{T}(2l,l,l-1)$ has $\gamma_{l} \geq l+1$ since $\gamma$ has at least one marked box at home and the configuration with all the marked boxes in the first $l$ boxes has no marked boxes at home. 

We proceed by strong induction on $l$. For $l=1$, $\mathcal{T}(2,1,0)$ there is only one configuration $\gamma=(2)$, for which $i=0$, and $T(1,0,0)=1$ since we have the configuration of length 1 and no marked boxes. 
For the induction step, we split into two cases:
\begin{itemize} 
\item Suppose $i=l-1$. We want to count the number of configurations in $\mathcal{T}(2l, l, l-1)$ with the last marked box in the last box. Note that in such a configuration, the last marked box is in the $(2l)^{\text{th}}$ box and so at home. Thus, deleting the last box of $\gamma$ decreases the length of the configuration by 1 and does not change the number of not at home marked boxes,  and we obtain a configuration of length $2l-1$ with $l-1$ marked boxes and $l-1$ of those not at home.  This process can be reversed in a bijective way, and so there are $T(2l-1,l-1,l-1)$ of such configurations. 

\item Suppose $i< l-1$. In this case, we define $\gamma'$ to be the configuration obtained by deleting the last marked box (only the mark, not the box) and the last two boxes of $\gamma$. Note that we only delete two boxes in total.  

Now, we need to look at the second-to-last marked box in $\gamma$, which we assume is in the $j^{\text{th}}$ position. Then, $\gamma' \in \mathcal{T}(2(l-1), l-1, l-2)$ with last marked box in the $(l+i'+1)^{\text{th}}$ box, for some $i'=j-k\geq 0$. By inductive hypothesis, there are $T(l-1+i', i',i')$ of those configurations. Moreover, there is a bijection between the set of configurations $\gamma'$ in $\mathcal{T}(2(l-1), l-1, l-2)$ with the last marked box in the $(l+i'+1)^{\text{th}}$ box and the set of configurations $\gamma$ in $\mathcal{T}(2l,l,l-1)$ with the last marked box in the $(l+i+1)^{\text{th}}$ box and the second-to-last marked box in the $(l+i'+1)^{\text{th}}$ box.

Therefore, the number of configurations in $\mathcal{T}(2l, l, l-1)$ with last marked box in the $(l+i+1)^{\text{th}}$ box by its second-to-last marked box is equal to
$$
\sum_{i'=0}^{i}T(l-1+i',i',i').
$$

By Proposition~\ref{thm: recursion of T}, we have that 
\begin{align*}
\sum_{i'=0}^{i}T(l-1+i',i',i') &= T(l-1,0,0) + 
\sum_{i'=1}^{i}  \left[T(l+i',i',i') - T(l-1+i',i'-1,i'-1)\right] \\
& = T(l+i,i,i),
\end{align*}
where we use that $T(l-1,0,0)=T(l,0,0)=1$, and that it is a telescopic sum. Thus, the result follows for this case too. 
\end{itemize}

\end{proof}

Using Theorem~\ref{thm: exactly k + l + 1} and the recursive formula from Proposition~\ref{thm: recursion of T} we can get and simplify the following telescopic sum:
\begin{align*}
T(2l,l,l-1) &= \sum_{i=0}^{l-1} T(l+i,i,i) = 
T(l+1,0,0) + \sum_{i=1}^{l-1} \left[ T(l+i+1, i,i) - T(l+i, i-1,i-1)\right] \\
&= T(2l,l-1,l-1),
\end{align*}
Thus, we get the independence of the second parameter in the special case as stated below.
\begin{cor}\label{cor: T(n,s,s) = T(n,s+1,s) special case}
    For $l>0$, $T(2l,l,l-1) = T(2l,l-1,l-1)$.
\end{cor}

\begin{remark}
    For the rest of this section, when $s=\left\lfloor \dfrac{n}{2}\right\rfloor$, we assume that $i=\left\lfloor \dfrac{n}{2}\right\rfloor$ and write $\mathcal{T}\left(n,\left\lfloor \dfrac{n}{2}\right\rfloor\right)$ and $T\left(n,\left\lfloor \dfrac{n}{2}\right\rfloor\right)$ instead.
\end{remark}

For the general case, we start with one more definition regarding the concatenation of two configurations and a technical result. 

Given $\gamma \in \mathcal{T}(2s,s,i)$ and $\beta\in \mathcal{T}(n',s',i')$, we define \demph{$\gamma \cdot \beta$} as the configuration obtained by concatenating the boxes of $\beta$ after the boxes of $\gamma$. The resulting configuration has length $n+n'$ and $s+s'$ marked boxes in total. Moreover, we know that the at home marked boxes of $\gamma$ continue to be at home in $\gamma \cdot \beta$. In general, we can not know what marked boxes of $\gamma$ continue to be at home in $\gamma\cdot \beta$. However, the following result studies the particular case we need.

\begin{example}
In Figure~\ref{fig:concatenation}, we illustrate an example of how to concatenate a configuration $\gamma\in \mathcal{T}(4,2,1)$ and $\beta\in \mathcal{T}(7,2,1)$, so that $\gamma\cdot\beta \in \mathcal{T}(11,4,2)$.

\begin{figure}[H]
    \centering
    \scalebox{0.5}{
\begin{tikzpicture}
    \fill[blue!30] (9,0) rectangle (10,1);
    \fill[blue!30] (11,0) rectangle (12,1);
    \fill[blue!30] (13,0) rectangle (14,1);
    \draw[step=1cm,black,very thin] (9,0) grid (16,1);
    \node at (10.5,0.5) {\scalebox{3}{$\times$}};
    \node at (15.5, 0.5) {\scalebox{3}{$\times$}};
    \node at (8.25,0.5) {\scalebox{2}{$\beta=$}};
    
    \fill[blue!30] (0,0) rectangle (1,1);
    \fill[blue!30] (2,0) rectangle (3,1);
    \draw[step=1cm,black,very thin] (0,0) grid (4,1);
    \node at (0.5,0.5) {\scalebox{3}{$\times$}};
    \node at (3.5, 0.5) {\scalebox{3}{$\times$}};
    \node at (-0.75,0.5) {\scalebox{2}{$\gamma=$}};

    \fill[blue!30] (1,-2.5) rectangle (2,-1.5);
    \fill[blue!30] (3,-2.5) rectangle (4,-1.5);
    \fill[blue!30] (5,-2.5) rectangle (6,-1.5);
    \fill[blue!30] (7,-2.5) rectangle (8,-1.5);
    \fill[blue!30] (9,-2.5) rectangle (10,-1.5);
    \foreach \x in {1,...,12} {
  \draw[black,very thin] (\x,-2.5) rectangle (\x+1,-1.5);}
    \node at (1.5,-2) {\scalebox{3}{$\times$}};
    \node at (4.5, -2) {\scalebox{3}{$\times$}};
    \node at (6.5,-2) {\scalebox{3}{$\times$}};
    \node at (11.5, -2) {\scalebox{3}{$\times$}};
    \node at (-0.75,-2) {\scalebox{2}{$\gamma \cdot \beta=$}};
\end{tikzpicture}}
    \caption{Illustration of the concatenation operation for certain configurations.}
    \label{fig:concatenation}
\end{figure}
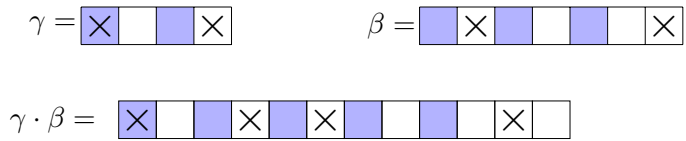
\end{example}

\begin{lemma}\label{lemma: gluing together configurations}
Let $\gamma \in \mathcal{T}(2s,s,i)$ and $\beta \in \mathcal{T}(n, s', i')$, with $0\leq i \leq s$, $0\leq i'\leq s'$ and $s'\leq \left\lfloor \frac n2\right\rfloor$. Then, $\gamma\cdot \beta \in \mathcal{T}(n+2s, s+s', i+i')$.   
\end{lemma}

\begin{proof}
By definition, a marked box in $\gamma$ is at home if and only if it is at home in $\gamma\cdot \beta$.
We want to show that the same statement is true for $\beta$. That is, a marked box in $\beta$ is at home if and only if it is at home in $\gamma \cdot \beta$. 
Suppose that the $m^{\text{th}}$ marked box in $\beta$ is at home, and so $\beta_m = 2m$. Then, this marked box corresponds to the $(s+m)^{\text{th}}$ marked box in $\gamma \cdot \beta$ and so $(\gamma \cdot \beta)_{s+m}=2s+2m$. Thus, it is also at home in $\gamma \cdot \beta$. For the other direction, suppose that the $m^{\text{th}}$ marked box in $\gamma \cdot \beta$ is at home, with $m=s+m'$ for some $m'\geq 1$ (so that it's in the $\beta$ area). Then, $(\gamma\cdot \beta)_{m} = 2m = 2s +2m'$. Now, ignoring the part corresponding to $\gamma$, this is the $(m')^{\text{th}}$ marked box in $\beta$, and it is in the $(2m')^{\text{th}}$ box, meaning that it is an at home marked box in $\beta$. This shows that the $nat$ statistic is preserved in this case. 
\end{proof}

Now we are ready to prove the independence of $s$ in the general case. 

\begin{thm}\label{prop: reduce to $T(n,s,s) = T(n,s+1,s)$}
For $n$, $s$, and $i$ such that $i<s < \lfloor \frac{n}{2}\rfloor$, 
$$T(n,s,i) = T(n,s+1,i) \qquad \text{ and } \qquad T(n,s,s) = T(n,s+1,s).$$
\end{thm}

\begin{proof}
We start by noticing that in Corollary~\ref{cor: T(n,s,s) = T(n,s+1,s) special case} we prove that $T(n,s,s) = T(n,s+1,s)$ for $n$ is even and $s = \frac{n}{2}-1$. We prove the claim holds in general by induction on $n$.

 For $n=2$, we have that $T(2,0,0) = T(2,1,0) = 1$ and the statement follows. Note also that for $i=0$, $T(n,s',0)=1$ for any $s' \leq \lfloor \frac{n}{2}\rfloor$, as there is always only one possible configuration with all marks at home.
For the inductive step, assume that the statement is true for all $l<n$. Applying Proposition~\ref{thm: recursion of T} twice,  we have that
$$T(n,s,s) = T(n-1,s,s) + T(n-1,s-1,s-1) = T(n-1,s+1,s) + T(n-1,s,s-1) = T(n,s+1,s),$$ 
where the middle equality follows from the inductive hypothesis, and the base case result follows.

Now, we proceed by strong induction, and so we assume that 
$T(k,s',s') = T(k,s'+1,s')$ for all $k < n$ and $s' < \lfloor \frac k2 \rfloor$. This implies that there exists a bijection from $\mathcal{T}(k,s',s')$ to $T(k,s'+1,s')$. For each pair $(k,s')$, we fix one such bijection and denote it by $\varphi_{k,s'}$. We construct $\phi$ from $\mathcal{T}(n,s,i)$ to $\mathcal{T}(n,s+1,i)$ in terms of these $\varphi_{k,s'}$ maps. Given $\gamma\in \mathcal{T}(n,s,i)$, we split it into two configurations $\gamma_l$ and $\gamma_r$, where $\gamma_l$ is the configuration containing every box in $\gamma$ up until the block that contains the last marked box at home and $\gamma_r$ is the remaining of the configuration. Then, if the last marked box at home is the $j^{\text{th}}$ marked box, we have that $\gamma_l\in\mathcal{T}(2j, j, j-s+i)$ and $\gamma_r\in \mathcal{T}(n-2j,s-j, s-j)$.
Then, we define $\gamma_r'$ as the image of $\gamma_r$ by its bijection $\varphi$, $\gamma_r'= \varphi_{n-2j, s-j}(\gamma_r)$, so that $\gamma_r' \in \mathcal{T}(n-2j, s-j+1,s-j)$. Finally, we define $\phi(\gamma)=\gamma_l \cdot \gamma_r'$. By Lemma~\ref{lemma: gluing together configurations}, $\phi(\gamma)\in \mathcal{T}(n, s+1, i)$. 

The details about the inverse map are left to the reader.
\end{proof}

Thus, we have the following consequence.
\begin{cor}
For $1\leq i \leq s \leq \left\lfloor \frac{n-1}{2} \right\rfloor$, $T(n,s,i) = T\left(n,i,i\right)$.
\end{cor}

\begin{remark}
For the rest of this section, we assume that $s=i$, and so we abbreviate $\mathcal{T}\left(n,s,i\right)$ and $T(n,s,i)$ to $\mathcal{T}(n,i)$ and $T(n,i)$, respectively.
\end{remark}

We finish this subsection by rewriting the recursive formula using the abbreviated notation.
\begin{cor}
For $1\leq i \leq \left\lfloor \frac{n-1}{2} \right\rfloor$, 
$$T(n,i) = T(n-1,i) + T(n-1,i-1) \qquad \text{ with } \qquad T(n,0) = T(n-1,0)=1.$$ 
\end{cor}

\begin{proof}
The case when $i=0$ follows by definition. For $1\leq i \leq \left\lfloor \frac{n-1}{2} \right\rfloor$, we have that 
$$T(n,i) = T(n,i,i) = T(n-1,i,i) + T(n-1,i-1,i-1) = T(n-1,i) + T(n-1,i-1).$$
\end{proof}

\subsection{Proof of Theorem~\ref{thm: closed form of $T(n,i)$}}
Let's start by recalling the statement of Theorem~\ref{thm: closed form of $T(n,i)$}.
\begin{thm*}
Let $s\leq \left\lfloor \frac n2 \right\rfloor$. For $i=0$, $T(n,s,0,b_0) = 1$, and for $0< i \leq s$, 
$$
T(n,s,i,b_0) = \binom{n}{i}-\binom{n}{i-1}.
$$
\end{thm*}

Using the results from Sections~\ref{subsect: independency of b_0} and~\ref{subsect: independency of s}, Theorem~\ref{thm: closed form of $T(n,i)$} is equivalent to the following result.
\begin{thm}\label{thm: abbreviated closed form of $T(n,i)$}
For $1\leq i\leq \left\lfloor \frac n2 \right\rfloor$, $$T(n,i) =\binom{n}{i}-\binom{n}{i-1} \qquad \text{ and } \qquad  T(n,0)=1.$$
\end{thm}
\begin{proof}
For $i=0$, $\mathcal{T}(n,0)$ has exactly one configuration, the one with every marked box at home. 

For $0<i\leq \left\lfloor \frac n2 \right\rfloor$ and $n\neq 2i$, we proceed by induction on $n$. 
The base case is $n=3$, for which we have that $T(3,1) = 2 = \binom{3}{1}-\binom{3}{0}$.

Now, assume the statement is true for all $\ell <n$. By Proposition~\ref{thm: recursion of T}, we have that
$$T(n,i) = T(n-1,i-1) + T(n-1,i) = \binom{n-1}{i-1} - \binom{n-1}{i-2} + \binom{n-1}{i} - \binom{n-1}{i-1}.$$ 
Reordering these terms according to their sign and using the recursive formula for binomial coefficients, we have that 
$$T(n,i) = \left[\binom{n-1}{i} + \binom{n-1}{i-1}\right] - \left[\binom{n-1}{i-1} + \binom{n-1}{i-1}\right] = \binom{n}{i} - \binom{n}{i-1}.$$

For $n=2i$, we know that $T(2i,i) = T(2i-1,i-1)$ by Theorem~\ref{thm: recursion of T- particular case}, and so by inductive hypothesis, we get that
\begin{align*}
    T(2i,i) &= T(2i-1,i-1) = \binom{2i-1}{i-1} - \binom{2i-1}{i-2}.
\end{align*}
Applying the symmetry of the binomial coefficients and the recursive formula for binomial coefficients, we have that 
\begin{align*}
    T(2i,i) &= \binom{2i-1}{i-1} + \binom{2i-1}{i} - \binom{2i-1}{i} - \binom{2i-1}{i-2} \\
    &=\binom{2i-1}{i} + \binom{2i-1}{i-1} -\left[ \binom{2i-1}{i-1} + \binom{2i-1}{i-2} \right]= \binom{2i}{i} - \binom{2i}{i-1}.
\end{align*}
Thus, the result follows.
\end{proof}

\printbibliography

\end{document}